\documentclass[10pt, reqno]{amsart}
\usepackage{amsmath, amsthm, amscd, amsfonts, amssymb, graphicx, color}
\usepackage[bookmarksnumbered, colorlinks, plainpages]{hyperref}
\hypersetup{colorlinks=true,linkcolor=red, anchorcolor=green, citecolor=cyan, urlcolor=red, filecolor=magenta, pdftoolbar=true}

\newtheorem{theorem}{Theorem}[section]
\newtheorem{lemma}[theorem]{Lemma}

\newtheorem{corollary}[theorem]{Corollary}
\theoremstyle{definition}
\newtheorem{definition}[theorem]{Definition}

\theoremstyle{remark}
\newtheorem{remark}[theorem]{Remark}

\numberwithin{equation}{section}

\begin{document}

\setcounter{page}{1}

\title[Generalized fractional integrals associated with operators]{Estimates for generalized fractional integrals associated with operators on Morrey--Campanato spaces}

\author[Cong Chen and Hua Wang]{Cong Chen and Hua Wang$^{*}$}

\address{School of Mathematics and Systems Science, Xinjiang University, Urumqi 830046, P. R. China.}
\email{\textcolor[rgb]{0.00,0.00,0.84}{wanghua@pku.edu.cn}}

\dedicatory{In memory of Li Xue}



\date{\today}

\begin{abstract}
Let $\mathcal{L}$ be the infinitesimal generator of an analytic semigroup $\big\{e^{-t\mathcal L}\big\}_{t>0}$ satisfying the Gaussian upper bounds. For given $0<\alpha<n$, let $\mathcal L^{-\alpha/2}$ be the generalized fractional integral associated with $\mathcal{L}$, which is defined as
\begin{equation*}
\mathcal L^{-\alpha/2}(f)(x):=\frac{1}{\Gamma(\alpha/2)}\int_0^{+\infty} e^{-t\mathcal L}(f)(x)t^{\alpha/2-1}dt,
\end{equation*}
where $\Gamma(\cdot)$ is the usual gamma function. For a locally integrable function $b(x)$ defined on $\mathbb R^n$, the related commutator operator $\big[b,\mathcal L^{-\alpha/2}\big]$ generated by $b$ and $\mathcal{L}^{-\alpha/2}$ is defined by
\begin{equation*}
\big[b,\mathcal L^{-\alpha/2}\big](f)(x):=b(x)\cdot\mathcal{L}^{-\alpha/2}(f)(x)-\mathcal{L}^{-\alpha/2}(bf)(x).
\end{equation*}
A new class of Morrey--Campanato spaces associated with $\mathcal{L}$ is introduced in this paper. The authors establish some new estimates for the commutators $\big[b,\mathcal L^{-\alpha/2}\big]$ on Morrey--Campanato spaces. The corresponding results for higher-order commutators$\big[b,\mathcal L^{-\alpha/2}\big]^m$($m\in \mathbb{N}$) are also discussed.
\end{abstract}

\maketitle

\section{Introduction and preliminaries}
\label{sec1}
Let $\mathcal{L}$ be the infinitesimal generator of an analytic semigroup $\big\{e^{-t\mathcal L}\big\}_{t>0}$ on $L^2(\mathbb R^n)$ with Gaussian upper bounds on its heat kernel, and suppose that $\mathcal{L}$ has a bounded holomorphic functional calculus on $L^2(\mathbb R^n)$. Let $\mathbb R^n$ be the $n$-dimensional Euclidean space endowed with the Lebesgue measure $dx$ and the Euclidean norm $|\cdot|$. Then $\mathcal{L}$ is the linear operator which generates an analytic semigroup $\big\{e^{-t\mathcal L}\big\}_{t>0}$ with kernel $\mathcal{P}_t(x,y)$ satisfying
\begin{equation*}
e^{-t\mathcal L}(f)(x):=\int_{\mathbb R^n}\mathcal{P}_t(x,y)f(y)\,dy, \quad t>0,
\end{equation*}
and there exist two positive constants $C$ and $A$ such that for all $x,y\in\mathbb R^n$ and all $t>0$, we have
\begin{equation}\label{G}
\big|\mathcal{P}_t(x,y)\big|\leq\frac{C}{t^{n/2}}\cdot e^{-A\frac{|x-y|^2}{t}}.
\end{equation}
For any $0<\alpha<n$, the generalized fractional integral $\mathcal L^{-\alpha/2}$ associated with the operator $\mathcal{L}$ is defined by
\begin{equation}\label{gefrac}
\mathcal L^{-\alpha/2}(f)(x):=\frac{1}{\Gamma(\alpha/2)}\int_0^{+\infty}e^{-t\mathcal L}(f)(x)t^{\alpha/2-1}dt.
\end{equation}
Let $\Delta$ be the Laplacian operator on $\mathbb R^n$, that is,
\begin{equation*}
\Delta:=\frac{\partial^2}{\partial x_1^2}+\cdots+\frac{\partial^2}{\partial x_n^2}.
\end{equation*}
Note that if $\mathcal L=-\Delta$ is the Laplacian on $\mathbb R^n$, then $\mathcal L^{-\alpha/2}$ is exactly the classical fractional integral operator $I_{\alpha}$ of order $\alpha$($0<\alpha<n$), which is given by
\begin{equation*}
I_{\alpha}(f)(x):=\frac{1}{\gamma(\alpha)}\int_{\mathbb R^n}\frac{f(y)}{|x-y|^{n-\alpha}}dy,\quad x\in\mathbb R^n,
\end{equation*}
where $\gamma(\alpha):=\frac{2^{\alpha}\pi^{n/2}\Gamma(\alpha/2)}{\Gamma({(n-\alpha)}/2)}$ and $\Gamma(\cdot)$ being the usual gamma function. Let $(-\Delta)^{\alpha/2}$(under $0<\alpha<n$) denote the $\alpha/2$-th order Laplacian operator. Then $u=I_{\alpha}f$ is viewed as a solution of the $\alpha/2$-th order Laplace equation
\begin{equation*}
(-\Delta)^{\alpha/2}u=f
\end{equation*}
in the sense of the Fourier transform; i.e., $(-\Delta)^{\alpha/2}$ exists as the inverse of $I_{\alpha}$.It is well known that the classical fractional integral operator $I_{\alpha}$ of order $\alpha$ plays an important role in harmonic analysis, potential theory and PDEs, particularly in the study of smoothness properties of functions. Let $0<\alpha<n$ and $1<p<q<\infty$. The classical Hardy--Littlewood--Sobolev theorem states that $I_{\alpha}$ is bounded from $L^p(\mathbb R^n)$ to $L^q(\mathbb R^n)$ if and only if $1/q=1/p-\alpha/n$.
Since the semigroup $\big\{e^{-t\mathcal L}\big\}_{t>0}$ has a kernel $\mathcal{P}_t(x,y)$ which satisfies the Gaussian upper bound \eqref{G}, it is easy to check that for all $x\in\mathbb R^n$,
\begin{equation}\label{dominate1}
\big|\mathcal L^{-\alpha/2}(f)(x)\big|\le C\cdot I_{\alpha}(|f|)(x).
\end{equation}
In fact, if we denote the kernel of $\mathcal L^{-\alpha/2}$ by $\mathcal K_{\alpha}(x,y)$, then it follows immediately from \eqref{gefrac} and Fubini's theorem that (see \cite{duong1} and \cite{mo})
\begin{equation}\label{kgamma}
\mathcal K_{\alpha}(x,y)=\frac{1}{\Gamma(\alpha/2)}\int_0^{+\infty}\mathcal{P}_t(x,y)t^{\alpha/2-1}dt,
\end{equation}
where $\mathcal{P}_t(x,y)$ is the kernel of $e^{-t\mathcal L}$. Thus, by using the Gaussian upper bound \eqref{G} and the expression \eqref{kgamma}, we can deduce that
\begin{align}\label{kernelk}
\big|\mathcal K_{\alpha}(x,y)\big|
&\leq\frac{1}{\Gamma(\alpha/2)}\int_0^{+\infty}\big|\mathcal{P}_t(x,y)\big|t^{\alpha/2-1}dt\notag\\
&\leq C\cdot\int_0^{+\infty} e^{-A\frac{|x-y|^2}{t}}\cdot t^{\alpha/2-n/2-1}dt\notag\\
&\leq C\cdot\frac{1}{|x-y|^{n-\alpha}}\int_0^{+\infty} e^{-\nu}\cdot\nu^{n/2-\alpha/2-1}d\nu\notag\\
&\leq C\cdot\frac{1}{|x-y|^{n-\alpha}}.
\end{align}
This proves \eqref{dominate1} with $C>0$ independent of $f$ (see \cite{duong1} and \cite{mo}).

For $x_0\in\mathbb R^n$ and $r>0$, let $B(x_0,r)=\{x\in\mathbb R^n:|x-x_0|<r\}$ denote the open ball centered at $x_0$ of radius $r$, $B(x_0,r)^{\complement}$ denote its complement and $m(B(x_0,r))$ be the Lebesgue measure of the ball $B(x_0,r)$. Recall that, for any given $1\leq p<\infty$, the space $L^p(\mathbb R^n)$ is defined as the set of all Lebesgue measurable functions $f$ on $\mathbb R^n$ such that
\begin{equation*}
\big\|f\big\|_{L^p}:=\bigg(\int_{\mathbb R^n}|f(x)|^p\,dx\bigg)^{1/p}<+\infty.
\end{equation*}
Let $L^{\infty}(\mathbb R^n)$ denote the Banach space of all essentially bounded measurable functions $f$ on $\mathbb R^n$.
The norm of $f\in L^{\infty}(\mathbb R^n)$ is given by
\begin{equation*}
\big\|f\big\|_{L^\infty}:=\underset{x\in\mathbb R^n}{\mbox{ess\,sup}}\,|f(x)|<+\infty.
\end{equation*}
A locally integrable function $f$ is said to belong to the space $\mathrm{BMO}(\mathbb R^n)$(bounded mean oscillation space, see \cite{john}), if
\begin{equation*}
\big\|f\big\|_{\mathrm{BMO}}:=\sup_{\mathcal{B}\subset\mathbb R^n}
\frac{1}{m(\mathcal{B})}\int_{\mathcal{B}}|f(x)-f_{\mathcal{B}}|\,dx<+\infty,
\end{equation*}
where $f_{\mathcal{B}}$ denotes the mean value of $f$ on the ball $\mathcal{B}$; i.e.,
\begin{equation*}
f_{\mathcal{B}}:=\frac{1}{m(\mathcal{B})}\int_{\mathcal{B}} f(y)\,dy
\end{equation*}
and the supremum is taken over all balls $\mathcal{B}$ in $\mathbb R^n$. Modulo constants, the space $\mathrm{BMO}(\mathbb R^n)$ is a Banach space with respect to the BMO norm $\|\cdot\|_{\mathrm{BMO}}$.

Let $b(x)$ be a locally integrable function on $\mathbb R^n$ and $0<\alpha<n$. Then the commutator operator generated by $b$ and $\mathcal{L}^{-\alpha/2}$ is defined by
\begin{equation}
\big[b,\mathcal L^{-\alpha/2}\big](f)(x):=b(x)\cdot\mathcal{L}^{-\alpha/2}(f)(x)-\mathcal{L}^{-\alpha/2}(bf)(x).
\end{equation}
The function $b$ is called \emph{the symbol function} of $\big[b,\mathcal L^{-\alpha/2}\big]$. The commutator $\big[b,\mathcal L^{-\alpha/2}\big]$ was first introduced and studied by Duong and Yan in \cite{duong1}. When $\mathcal L=-\Delta$ is the Laplacian operator on $\mathbb R^n$, the commutator $\big[b,\mathcal L^{-\alpha/2}\big]=\big[b,I_{\alpha}\big]$ generated by $b$ and $I_{\alpha}$ was first defined by Chanillo in \cite{cha}. Then for all $0<\alpha<n$,
\begin{equation*}
\big[b,I_{\alpha}\big](f)(x)=\frac{1}{\gamma(\alpha)}\int_{\mathbb R^n}\frac{[b(x)-b(y)]}{|x-y|^{n-\alpha}}f(y)\,dy,
\end{equation*}
and
\begin{equation*}
\big[b,\mathcal L^{-\alpha/2}\big](f)(x)=\int_{\mathbb R^n}\big[b(x)-b(y)\big]\mathcal{K}_{\alpha}(x,y)f(y)\,dy,
\end{equation*}
where $\mathcal{K}_{\alpha}(x,y)$ denotes the kernel of $\mathcal{L}^{-\alpha/2}$.

In the present paper, we will study the boundedness of generalized fractional integral operators and commutators.

\begin{remark}
The property \eqref{G} is satisfied by a large class of differential operators, such as (magnetic) Schr\"{o}dinger operators and second-order elliptic operators of divergence form, see \cite{duong1,duong3,duong2} for more details.
\end{remark}

For the classical fractional integral $I_{\alpha}$ and the commutator $\big[b,I_{\alpha}\big]$ acting on Lebesgue spaces, we have
\begin{theorem}\label{11}
The following statements are true:
\begin{enumerate}
  \item Let $0<\alpha<n$, $1<p<n/{\alpha}$ and $1/q=1/p-\alpha/n$. Then the classical fractional integral $I_{\alpha}$ is bounded from $L^{p}(\mathbb R^n)$ to $L^{q}(\mathbb R^n)$.
  \item If $b\in \mathrm{BMO}(\mathbb R^n)$, then the commutator operator $\big[b,I_{\alpha}\big]$ is bounded from $L^{p}(\mathbb R^n)$ to $L^{q}(\mathbb R^n)$.
\end{enumerate}
\end{theorem}
For the proof of this result, see, for example, Stein \cite[Chapter V]{stein}, Grafakos \cite[Chapter 6]{grafakos} and \cite[Theorem 1]{cha}.
For the weighted version of this result, see also Muckenhoupt--Wheeden \cite{muckenhoupt1,muckenhoupt2}, Segovia--Torrea \cite[Theorem 2.3]{segovia} and Lu--Ding--Yan \cite[Chapter 3]{lu}.

The result in Theorem \ref{11} can be extended from $-\Delta$ to the more general operator $\mathcal{L}$ with Gaussian upper bound \eqref{G}. In \cite{duong1}, by using a new sharp maximal function, $M^{\#}_{\mathcal{L}}$, adapted to the semigroup $\big\{e^{-t\mathcal L}\big\}_{t>0}$(see Definition \ref{martel} below), Duong and Yan studied unweighted estimates for commutators of generalized fractional integrals, and proved that for all $0<\alpha<n$ and $b\in \mathrm{BMO}(\mathbb R^n)$, both $\mathcal{L}^{-\alpha/2}$ and $\big[b,\mathcal L^{-\alpha/2}\big]$ are all bounded from $L^{p}(\mathbb R^n)$ to $L^{q}(\mathbb R^n)$. By the pointwise inequality \eqref{kernelk} and the kernel estimate for the difference operator $\mathcal{L}^{-\alpha/2}-e^{-t\mathcal L}\mathcal{L}^{-\alpha/2}$ in \cite{deng,duong1}(with $0<\alpha<n$), we have
\begin{theorem}\label{12}
The following statements are true:
\begin{enumerate}
  \item Let $0<\alpha<n$, $1<p<n/{\alpha}$ and $1/q=1/p-\alpha/n$. Then the generalized fractional integral operator $\mathcal L^{-\alpha/2}$ is bounded from $L^{p}(\mathbb R^n)$ to $L^{q}(\mathbb R^n)$.
  \item If $b\in \mathrm{BMO}(\mathbb R^n)$, then the commutator operator $\big[b,\mathcal L^{-\alpha/2}\big]$ is bounded from $L^{p}(\mathbb R^n)$ to $L^{q}(\mathbb R^n)$.
\end{enumerate}
\end{theorem}
For the proof of this result, see Duong--Yan \cite[Theorem 1.1 and Lemma 2.2]{duong1}. Another proof was given by Cruz-Uribe--Martell--P\'erez in \cite[Proposition 3.2]{cruz}, which used a variant of the $A_p$ extrapolation theorem and fractional Orlicz maximal operators. For the weighted version of this result, see Auscher--Martell \cite[Theorems 1.3 and 1.4]{auscher}. See also \cite[Section 5.4]{benyi} for another proof of such weighted estimates.

For $0<\beta\leq1$, we say that a real-valued function $f$ on $\mathbb R^n$ is a Lipschitz function of order $\beta$, if
\begin{equation*}
\big\|f\big\|_{\mathrm{Lip}_{\beta}}:=\sup_{x,y\in\mathbb R^n,x\neq y}\frac{|f(x)-f(y)|}{|x-y|^{\beta}}<+\infty.
\end{equation*}
Let $\mathrm{Lip}_{\beta}(\mathbb R^n)$ be the space of Lipschitz functions of order $\beta$, that is,
\begin{equation*}
\mathrm{Lip}_{\beta}(\mathbb R^n)=\big\{f:\|f\|_{\mathrm{Lip}_{\beta}}<+\infty\big\}.
\end{equation*}
If $b\in \mathrm{Lip}_{\beta}(\mathbb R^n)$ with $0<\beta\leq1$ and $0<\alpha+\beta<n$, then by the pointwise inequality \eqref{kernelk} and the definition of $\mathrm{Lip}_{\beta}(\mathbb R^n)$, we can deduce that for any $x\in\mathbb R^n$,
\begin{equation}\label{point1}
\begin{split}
\Big|\big[b,I_{\alpha}\big](f)(x)\Big|
&\leq\big\|b\big\|_{\mathrm{Lip}_{\beta}}\frac{1}{\gamma(\alpha)}
\int_{\mathbb R^n}\frac{|f(y)|}{|x-y|^{n-\alpha-\beta}}dy\\
&=\big\|b\big\|_{\mathrm{Lip}_{\beta}}I_{\alpha+\beta}(|f|)(x),
\end{split}
\end{equation}
and
\begin{equation}\label{point2}
\begin{split}
\Big|\big[b,\mathcal L^{-\alpha/2}\big](f)(x)\Big|
&\leq\int_{\mathbb R^n}\big|b(x)-b(y)\big|\cdot|\mathcal{K}_{\alpha}(x,y)||f(y)|\,dy\\
&\leq C\big\|b\big\|_{\mathrm{Lip}_{\beta}}\int_{\mathbb R^n}\frac{|f(y)|}{|x-y|^{n-\alpha-\beta}}dy\\
&\leq C\big\|b\big\|_{\mathrm{Lip}_{\beta}}I_{\alpha+\beta}(|f|)(x).
\end{split}
\end{equation}
Thus, by Theorem \ref{11}, we can prove that the commutators $\big[b,I_{\alpha}\big]$ and $\big[b,\mathcal L^{-\alpha/2}\big]$ are bounded operators from $L^{p}(\mathbb R^n)$ to $L^q(\mathbb R^n)$, whenever $1<p<n/{(\alpha+\beta)}$ and $1/q=1/p-{(\alpha+\beta)}/{n}$.
This result was obtained by Paluszy\'{n}ski in \cite{paluszynski} and Mo--Lu in \cite{mo}.

On the other hand, the classical Morrey spaces $\mathcal{M}^{p,\beta}(\mathbb R^n)$ were originally introduced by Morrey in \cite{morrey} to study the local regularity of solutions to second order elliptic partial differential equations. Nowadays these spaces have been studied intensively in the literature, and found a wide range of applications in harmonic analysis, potential theory and PDEs. Let us now recall the definition of the classical Morrey space. Let $1\leq p<\infty$ and $-n/p\leq\beta\leq0$. We denote by $\mathcal{M}^{p,\beta}(\mathbb R^n)$ the Morrey space of all $p$-locally integrable functions $f$ on $\mathbb R^n$ such that
\begin{equation*}
\begin{split}
\big\|f\big\|_{\mathcal{M}^{p,\beta}}:=&\sup_{\mathcal{B}\subset\mathbb R^n}\frac{1}{m(\mathcal{B})^{\beta/n}}
\bigg(\frac{1}{m(\mathcal{B})}\int_{\mathcal{B}}|f(x)|^p\,dx\bigg)^{1/p}\\
=&\sup_{\mathcal{B}\subset\mathbb R^n}\frac{1}{m(\mathcal{B})^{\beta/n+1/p}}
\big\|f\cdot\chi_{\mathcal{B}}\big\|_{L^p}<+\infty.
\end{split}
\end{equation*}
Note that $\mathcal{M}^{p,-n/p}(\mathbb R^n)=L^p(\mathbb R^n)$ and $\mathcal{M}^{p,0}(\mathbb R^n)=L^{\infty}(\mathbb R^n)$ by the Lebesgue
differentiation theorem. If $\beta<-n/p$ or $\beta>0$, then $\mathcal{M}^{p,\beta}(\mathbb R^n)=\Theta$, where $\Theta$ is the set of all functions equivalent to $0$ on $\mathbb R^n$.

For boundedness properties of the classical fractional integral $I_{\alpha}$ and the commutator $\big[b,I_{\alpha}\big]$ on Morrey spaces, we have
\begin{theorem}\label{13}
The following estimate holds:
\begin{enumerate}
  \item Let $0<\alpha<n$, $1<p<n/{\alpha}$ and $1/q=1/p-{\alpha}/n$. Suppose that $-n/p\leq\beta<(-\alpha)$, then the classical fractional integral $I_{\alpha}$ is bounded from $\mathcal{M}^{p,\beta}(\mathbb R^n)$ to $\mathcal{M}^{q,\alpha+\beta}(\mathbb R^n)$.
  \item If $b\in \mathrm{BMO}(\mathbb R^n)$, then the commutator operator $\big[b,I_{\alpha}\big]$ is bounded from $\mathcal{M}^{p,\beta}(\mathbb R^n)$ to $\mathcal{M}^{q,\alpha+\beta}(\mathbb R^n)$.
\end{enumerate}
\end{theorem}
For the proof of this result, see, for example, Adams \cite{adams,adams1} and Peetre \cite{peetre}. See also \cite{komori,wang} for the weighted case.

Motivated by Theorem \ref{12}, it is natural to ask whether the result in Theorem \ref{13} also holds for the generalized fractional integrals and related commutators. We give a positive answer to this question.

\begin{theorem}\label{generalizedmorrey}
The following estimate holds:
\begin{enumerate}
  \item Let $0<\alpha<n$, $1<p<n/{\alpha}$ and $1/q=1/p-{\alpha}/n$. Suppose that $-n/p\leq\beta<(-\alpha)$, then the generalized fractional integral operator $\mathcal{L}^{-\alpha/2}$ is bounded from $\mathcal{M}^{p,\beta}(\mathbb R^n)$ to $\mathcal{M}^{q,\alpha+\beta}(\mathbb R^n)$.
  \item If $b\in \mathrm{BMO}(\mathbb R^n)$, then the commutator operator $\big[b,\mathcal{L}^{-\alpha/2}\big]$ is bounded from $\mathcal{M}^{p,\beta}(\mathbb R^n)$ to $\mathcal{M}^{q,\alpha+\beta}(\mathbb R^n)$.
\end{enumerate}
\end{theorem}
The proof is based on the pointwise inequality \eqref{kernelk} and an estimate on the kernel of the
difference operator $\mathcal{L}^{-\alpha/2}-e^{-t\mathcal L}\mathcal{L}^{-\alpha/2}$ in \cite{deng,duong1}(with $0<\alpha<n$).A weighted version of this result has been obtained by the author in \cite[Theorem 1.4]{wang2}(see also\cite[Corollaries 18 and 24]{wang3}).
\begin{remark}
When $-n/p\leq\beta<-\alpha$ and $1/q=1/p-{\alpha}/n$, one can see that
\begin{equation*}
-n/q\leq \alpha+\beta<0.
\end{equation*}
\end{remark}
If $b\in \mathrm{Lip}_{\beta_1}(\mathbb R^n)$ with $0<\beta_1\leq1$ and $0<\alpha+\beta_1<n$, then by estimates \eqref{point1},\eqref{point2} and Theorem \ref{13}, the commutators $\big[b,I_{\alpha}\big]$ and $\big[b,\mathcal{L}^{-\alpha/2}\big]$ are bounded operators from $\mathcal{M}^{p,\beta}(\mathbb R^n)$ to $\mathcal{M}^{q,\gamma}(\mathbb R^n)$, where
\begin{equation*}
-\frac{n}{\,p\,}\leq\beta<-(\alpha+\beta_1),~~~\frac{\,1\,}{q}=\frac{\,1\,}{p}-\frac{\alpha+\beta_1}{n},~~~\mbox{and}~~~ \gamma=\alpha+\beta_1+\beta.
\end{equation*}
In the present situation, we see that
\begin{equation*}
\frac{\,1\,}{q}=\frac{\,1\,}{p}-\frac{\alpha+\beta_1}{n}<\frac{\,1\,}{p}\Longrightarrow q>p,
\end{equation*}
and
\begin{equation*}
\gamma=\beta+(\alpha+\beta_1)\Longrightarrow\gamma>\beta.
\end{equation*}
This result was established by the author in \cite[Theorems 1.5 and 1.6]{wang2}.

It is well known that the classical Morrey--Campanato spaces play an important role in the study of partial differential equations and harmonic analysis, see \cite{deng2,janson,shi2} for more details.

\begin{definition}\label{17}
Let $1\leq p<\infty$ and $-n/p\leq\beta\leq1$. A locally integrable function $f$ is said to belong to the Morrey--Campanato space $\mathcal{C}^{p,\beta}(\mathbb R^n)$, if
\begin{equation}\label{quantity1}
\big\|f\big\|_{\mathcal{C}^{p,\beta}}:=\sup_{\mathcal{B}\subset\mathbb R^n}
\frac{1}{m(\mathcal{B})^{\beta/n}}\bigg(\frac{1}{m(\mathcal{B})}
\int_{\mathcal{B}}\big|f(x)-f_{\mathcal{B}}\big|^{p}\,dx\bigg)^{1/{p}}<+\infty,
\end{equation}
where $f_{\mathcal{B}}$ denotes the mean value of $f$ and $m(\mathcal{B})$ is the Lebesgue measure of the ball $\mathcal{B}$. The quantity $\big\|f\big\|_{\mathcal{C}^{p,\beta}}$ is called the Morrey--Campanato norm of $f$.
\end{definition}
The space $\mathcal{C}^{p,\beta}(\mathbb R^n)$ was first introduced by Campanato in \cite{cam}, and was studied extensively in the literature. Both the space $\mathrm{BMO}(\mathbb R^n)$ and (homogenous) Lipschitz function space $\mathrm{Lip}_{\beta}(\mathbb R^n)$ are special cases of Morrey--Campanato spaces.
\begin{itemize}
  \item It is well known that when $\beta=0$ and $1\leq p<\infty$, then
\begin{equation*}
\mathcal{C}^{p,\beta}(\mathbb R^n)=\mathrm{BMO}(\mathbb R^n)
\end{equation*}
with equivalent norms, see \cite{deng,grafakos,lu} for example.
  \item If $1\leq p<\infty$ and $0<\beta\leq1$, then
\begin{equation}\label{equilip}
\mathcal{C}^{p,\beta}(\mathbb R^n)=\mathrm{Lip}_{\beta}(\mathbb R^n)
\end{equation}
with equivalent norms, see \cite{janson} and \cite{paluszynski} for example.
  \item We remark that the equation \eqref{equilip} allows us to give an integral characterization of the (homogeneous) Lipschitz function space $\mathrm{Lip}_{\beta}(\mathbb R^n)$. This fact further leads to a generalization of the classical Sobolev embedding theorem. It is also well known that $\mathcal{C}^{1,1/p-1}(\mathbb R^n)$ is the dual space of Hardy space $H^p(\mathbb R^n)$ when $0<p<1$, and $\mathcal{C}^{1,0}(\mathbb R^n)=\mathrm{BMO}(\mathbb R^n)$ is the dual space of Hardy space $H^1(\mathbb R^n)$, see \cite{grafakos} and \cite{janson} for further details.
  \item If $-n/p\leq\beta<0$ and $1\leq p<\infty$, then in this case, it can be shown that
\begin{equation*}
\mathcal{C}^{p,\beta}(\mathbb R^n)\supseteq \mathcal{M}^{p,\beta}(\mathbb R^n).
\end{equation*}
This inclusion relation can be found, for instance, in \cite{shi1,shi2}.
\end{itemize}

As mentioned above, many authors are interested in the study of commutators for which the symbol functions belong to BMO spaces and Lipschitz spaces. It is interesting to consider the boundedness of commutators with symbol functions belonging to $\mathcal{C}^{p_1,\beta_1}(\mathbb R^n)$, when $-n/{p_1}\leq\beta_1<0$ and $1\leq p_1<\infty$. It was proved by Shi and Lu that the commutator $\big[b,I_{\alpha}\big]$ is a bounded operator from $\mathcal{M}^{p,\beta}(\mathbb R^n)$ to $\mathcal{C}^{q,\gamma}(\mathbb R^n)$ for appropriate indices, under the assumption that $b\in \mathcal{C}^{p_1,\beta_1}(\mathbb R^n)$ with $-n/{p_1}\leq\beta_1<0$ and $1\leq p_1<\infty$. Moreover, some new characterizations of Morrey--Campanato spaces via the boundedness of $\big[b,I_{\alpha}\big]$ were also given, by new methods instead of the sharp maximal
function theorem. See \cite[Theorem 1.1]{shi2}. For the corresponding results of commutators associated with Calder\'{o}n--Zygmund singular integral operators, see \cite[Theorems 1.2 and 1.3]{shi1}.

The main purpose of this paper is to prove that the commutator $\big[b,\mathcal{L}^{-\alpha/2}\big]$ is bounded from $\mathcal{M}^{p,\beta}(\mathbb R^n)$ to $\mathcal{C}^{q,\gamma}_{\mathcal{L}}(\mathbb R^n)$ for suitable indices, which is an extension of the result of Shi--Lu. Here the Morrey--Campanato space $\mathcal{C}^{q,\gamma}_{\mathcal{L}}(\mathbb R^n)$ is defined for all functions $f$ with suitable bounds on growth. See Definition \ref{wangdefi} below.

\section{New Morrey--Campanato spaces associated with $\mathcal{L}$}
To study the operators $\mathcal{L}^{-\alpha/2}$ and $\big[b,\mathcal{L}^{-\alpha/2}\big]$, we will introduce the Morrey--Campanato space $\mathcal{C}^{p,\gamma}_{\mathcal{L}}(\mathbb R^n)$ associated with the operator $\mathcal{L}$. Let us first introduce some notations and
definitions related to the new Morrey--Campanato space.

A family of operators $\{\mathbf{A}_t\}_{t>0}$ is said to be ``generalized approximation to the identity", if for every $t>0$, $\mathbf{A}_t$ is represented by the kernel $p_t(x,y)$ that satisfies an upper bound
\begin{equation*}
\big|p_t(x,y)\big|\leq h_t(x,y):=t^{-n/2}g\Big(\frac{|x-y|}{\sqrt{t\,}}\Big),
\end{equation*}
for all $x,y\in\mathbb R^n$ and $t>0$. Here $g$ is a positive, bounded and decreasing function satisfying
\begin{equation}\label{15}
\lim_{r\to\infty}r^{n+\varepsilon}g(r)=0
\end{equation}
for some $\varepsilon>0$.

Let $\varepsilon$ be as in \eqref{15} and let $0<\varrho<\varepsilon$. For any $1\leq p<\infty$, a locally integrable function $f$ is said to be a function of $(p,\varrho)$-type, if it satisfies
\begin{equation}\label{16}
\bigg(\int_{\mathbb R^n}\frac{|f(x)|^p}{(1+|x|)^{n+\varrho}}dx\bigg)^{1/p}\leq C<+\infty,
\end{equation}
and we denote by $M_{(p,\varrho)}$ the collection of all functions of $(p,\varrho)$-type. The smallest bound $C$ satisfying the condition \eqref{16} is then taken to be the norm of $f$, and is denoted by $\|f\|_{M_{(p,\varrho)}}$. It is easy to see that $M_{(p,\varrho)}$ is a Banach function space under the norm $\|f\|_{M_{(p,\varrho)}}$. We set
\begin{equation*}
M_p:=\bigcup_{\varrho:0<\varrho<\varepsilon}M_{(p,\varrho)}.
\end{equation*}
For any $f\in L^p(\mathbb R^n)$ with $p\geq1$, the sharp maximal function $M^{\#}_{\mathbf{A}}(f)$ associated with ``generalized approximation to the identity" is defined as follows:
\begin{equation*}
M^{\#}_{\mathbf{A}}(f)(x)=\sup_{x\in \mathcal{B}}\frac{1}{m(\mathcal{B})}
\int_{\mathcal{B}}\big|f(y)-\mathbf{A}_{t_{\mathcal{B}}}f(y)\big|\,dy,
\end{equation*}
where $t_{\mathcal{B}}=r_{\mathcal{B}}^2$, $r_{\mathcal{B}}$ is the radius of the ball $\mathcal{B}$,
\begin{equation*}
\mathbf{A}_{t_{\mathcal{B}}}f(y)=\int_{\mathbb R^n}p_{t_{\mathcal{B}}}(y,z)f(z)\,dz,
\end{equation*}
and the supremum is taken over all balls $\mathcal{B}$ containing the point $x$. The sharp maximal function $M^{\#}_{\mathbf{A}}$ was first introduced and studied by Martell in \cite{martell}. We remark that our analytic semigroup $\big\{e^{-t\mathcal L}\big\}_{t>0}$ is ``generalized approximation to the identity". In particular, we give

\begin{definition}\label{martel}
For every $f\in L^p(\mathbb R^n)$ with $p\geq1$, the sharp maximal function $M^{\#}_{\mathcal{L}}(f)$ is defined by
\begin{equation*}
M^{\#}_{\mathcal{L}}(f)(x):=\sup_{x\in \mathcal{B}}\frac{1}{m(\mathcal{B})}
\int_{\mathcal{B}}\big|f(y)-e^{-t_{\mathcal{B}}\mathcal{L}}f(y)\big|\,dy,
\end{equation*}
where $t_{\mathcal{B}}=r_{\mathcal{B}}^2$, $r_{\mathcal{B}}$ is the radius of the ball $\mathcal{B}$ and the supremum is taken over all balls $\mathcal{B}$ in $\mathbb R^n$.
\end{definition}

\begin{definition}\label{newbmo}
Let $f\in L^p(\mathbb R^n)$ with $p\geq1$ and $f\in M_p$. We say that a function $f$ is in the space $\mathrm{BMO}_{\mathcal{L}}(\mathbb R^n)$ associated with the operator $\mathcal{L}$, if the sharp maximal function $M^{\#}_{\mathcal{L}}(f)\in L^{\infty}(\mathbb R^n)$, and we define
\begin{equation*}
\big\|f\big\|_{\mathrm{BMO}_{\mathcal{L}}}=\big\|M^{\#}_{\mathcal{L}}(f)\big\|_{L^{\infty}}.
\end{equation*}
\end{definition}

The new $\mathrm{BMO}_{\mathcal{L}}$ space associated with the operator $\mathcal{L}$ was introduced and studied by Duong and Yan in \cite{duong3}(see also \cite{duong2,deng2,deng}). The idea is that the quantity $e^{-t\mathcal{L}}(f)$ can be viewed as an average version of $f$(at the scale $t$) and the quantity $e^{-t_{\mathcal{B}}\mathcal{L}}f(x)$ can be used to replace the mean value $f_{\mathcal{B}}$ in the definition of the classical BMO space. Here $t_{\mathcal{B}}$ is equal to the square of the radius of $\mathcal{B}$. Notice that $\mathrm{BMO}_{\mathcal{L}}$ is a semi-normed vector space, with the semi-norm vanishing on the kernel space $\mathrm{Ker}_{\mathcal{L}}$ defined by
\begin{equation*}
\mathrm{Ker}_{\mathcal{L}}:=\Big\{f\in M_p:e^{-t\mathcal{L}}(f)=f,~~\mbox{for all}~~ t>0\Big\}.
\end{equation*}
The class of $\mathrm{BMO}_{\mathcal{L}}$ functions (modulo $\mathrm{Ker}_{\mathcal{L}}$) is a Banach function space.
\begin{itemize}
  \item A natural question arising from Definition \ref{newbmo} is to compare the classical BMO space and the $\mathrm{BMO}_{\mathcal{L}}$ space associated with the operator $\mathcal{L}$. Denote by $e^{t\Delta}$ the heat semigroup on $\mathbb R^n$. It can be shown that the classical BMO space (modulo all constant functions) and the $\mathrm{BMO}_{\Delta}$ space (modulo $\mathrm{Ker}_{\Delta}$) coincide, and their norms are equivalent, see \cite{duong3,duong2}.
  \item Assume that for every $t>0$, the equation
  \begin{equation*}
  e^{-t\mathcal{L}}(\mathbf{1})(x)=\mathbf{1}
  \end{equation*}
  holds for almost everywhere $x\in\mathbb R^n$, that is,
  \begin{equation*}
  \int_{\mathbb R^n}\mathcal{P}_t(x,y)\,dy=1
  \end{equation*}
  for almost all $x\in\mathbb R^n$. Then we have $\mathrm{BMO}(\mathbb R^n)\subset\mathrm{BMO}_{\mathcal{L}}(\mathbb R^n)$, and there exists a positive constant $C>0$ such that
  \begin{equation}\label{18}
  \big\|f\big\|_{\mathrm{BMO}_{\mathcal{L}}}\leq C\big\|f\big\|_{\mathrm{BMO}}.
  \end{equation}
  However, the converse inequality does not hold in general. See Martell \cite[Proposition 3.1]{martell} and Deng--Duong--Sikora--Yan \cite[Proposition 2.3]{deng}. As pointed out in \cite{deng}, the condition $e^{-t\mathcal{L}}(\mathbf{1})(x)=\mathbf{1}$ is also necessary for \eqref{18}. Indeed, it follows from \eqref{18} that $\big\|\mathbf{1}\big\|_{\mathrm{BMO}_{\mathcal{L}}}=0$. This in turn implies that for all $t>0$, $e^{-t\mathcal{L}}(\mathbf{1})(x)=\mathbf{1}$ holds for almost everywhere $x\in\mathbb R^n$.
\end{itemize}
For further details about the properties and applications of the spaces $\mathrm{BMO}_{\mathcal{L}}(\mathbb R^n)$, we refer the reader to \cite{deng2,deng,duong3,duong2} and the references therein.

Motivated by Definitions \ref{17} and \ref{newbmo}, we now introduce our space $\mathcal{C}^{p,\gamma}_{\mathcal{L}}(\mathbb R^n)$ associated with $\mathcal{L}$, by using the function $e^{-t_{\mathcal{B}}\mathcal{L}}f(x)$ to replace the average $f_{\mathcal{B}}$ in the Morrey--Campanato norm \eqref{quantity1}, and $t_{\mathcal{B}}$ equals the square of the radius of $\mathcal{B}$.

\begin{definition}\label{wangdefi}
Assume that the semigroup $\big\{e^{-t\mathcal L}\big\}_{t>0}$ has a kernel $\mathcal{P}_t(x,y)$ satisfying the Gaussian upper bound \eqref{G}.
Let $1\leq p<\infty$ and $-n/p\leq\gamma\leq1$. A locally integrable function $f\in M_p$ is said to belong to the Morrey--Campanato space $\mathcal{C}^{p,\gamma}_{\mathcal{L}}(\mathbb R^n)$ associated with the operator $\mathcal{L}$, if
\begin{equation*}
\big\|f\big\|_{\mathcal{C}^{p,\gamma}_{\mathcal{L}}}:=\sup_{\mathcal{B}\subset\mathbb R^n}
\frac{1}{m(\mathcal{B})^{\gamma/n}}\bigg(\frac{1}{m(\mathcal{B})}
\int_{\mathcal{B}}\big|f(x)-e^{-t_{\mathcal{B}}\mathcal{L}}f(x)\big|^{p}\,dx\bigg)^{1/{p}}<+\infty,
\end{equation*}
where $t_{\mathcal{B}}=r_{\mathcal{B}}^2$, $r_{\mathcal{B}}$ is the radius of the ball $\mathcal{B}$ and the supremum is taken over all balls $\mathcal{B}$ in $\mathbb R^n$.
\end{definition}
\begin{itemize}
  \item A natural question arising from Definition \ref{wangdefi} is to compare the classical Morrey--Campanato space $\mathcal{C}^{p,\gamma}(\mathbb R^n)$ with the new function space $\mathcal{C}^{p,\gamma}_{\mathcal{L}}(\mathbb R^n)$. Following the idea of \cite{deng2,duong3,duong2}, we can also show that when $\mathcal{L}$ is the Laplacian $\Delta$ in $\mathbb R^n$, then the corresponding space $\mathcal{C}^{p,\gamma}_{\mathcal{L}}(\mathbb R^n)$ coincides with the classical Morrey--Campanato space $\mathcal{C}^{p,\gamma}(\mathbb R^n)$, and their norms are equivalent.
  \item When $\gamma=0$ and $1\leq p<\infty$, then
  \begin{equation*}
  \mathcal{C}^{p,\gamma}_{\mathcal{L}}(\mathbb R^n)=\mathrm{BMO}_{\mathcal{L}}(\mathbb R^n)
  \end{equation*}
  with equivalent norms, since a variant of the John--Nirenberg inequality holds for functions in $\mathrm{BMO}_{\mathcal{L}}(\mathbb R^n)$. See \cite{deng2,duong3,duong2} for more details.
  \item Using a similar argument to that used in the proof of \cite[Proposition 3.1]{martell} and \cite[Proposition 2.3]{deng}, we can also show that $\mathcal{C}^{p,\gamma}(\mathbb R^n)$ is a subspace of $\mathcal{C}^{p,\gamma}_{\mathcal{L}}(\mathbb R^n)$, under the assumption that $\mathcal{L}$ satisfies a conservation property of the semigroup $e^{-t\mathcal{L}}(\mathbf{1})=\mathbf{1}$ for every $t>0$.
\end{itemize}
We have the following lemma.
\begin{lemma}\label{inclusion}
Let $1\leq p<\infty$ and $-n/p\leq\gamma\leq1$ with $\gamma\neq0$. Assume that the semigroup $\big\{e^{-t\mathcal L}\big\}_{t>0}$ has a kernel $\mathcal{P}_t(x,y)$ satisfying the Gaussian upper bound \eqref{G} and for every $t>0$, the equation
  \begin{equation*}
  e^{-t\mathcal{L}}(\mathbf{1})(x)=\mathbf{1}
  \end{equation*}
  holds for almost all $x\in\mathbb R^n$. Then we have
  $\mathcal{C}^{p,\gamma}(\mathbb R^n)\subset\mathcal{C}^{p,\gamma}_{\mathcal{L}}(\mathbb R^n)$, and there exists a positive constant $C>0$ such that
  \begin{equation}\label{110}
  \big\|f\big\|_{\mathcal{C}^{p,\gamma}_{\mathcal{L}}}\leq C\big\|f\big\|_{\mathcal{C}^{p,\gamma}}.
  \end{equation}
\end{lemma}
Here we give the proof for the sake of completeness.
\begin{proof}
Assume that $f\in \mathcal{C}^{p,\gamma}(\mathbb R^n)$ with $1\leq p<\infty$ and $-n/p\leq\gamma\leq1$ with $\gamma\neq0$. Let $\mathcal{B}=B(x_0,r_{\mathcal{B}})$ be a fixed ball centered at $x_0$ and with radius $r_{\mathcal{B}}$, and let $t_{\mathcal{B}}=r_{\mathcal{B}}^2$. Then we get
\begin{equation}\label{21}
\big|f(x)-e^{-t_{\mathcal{B}}\mathcal{L}}f(x)\big|\leq\big|f(x)-f_{\mathcal{B}}\big|+
\big|e^{-t_{\mathcal{B}}\mathcal{L}}f(x)-f_{\mathcal{B}}\big|.
\end{equation}
Note that by the assumption, we have
\begin{equation*}
\int_{\mathbb R^n}\mathcal{P}_{t_{\mathcal{B}}}(x,y)\,dy=1\Longrightarrow\int_{\mathbb R^n}\mathcal{P}_{t_{\mathcal{B}}}(x,y)\cdot f_{\mathcal{B}}\,dy=f_{\mathcal{B}}
\end{equation*}
for almost all $x\in\mathbb R^n$. Thus,
\begin{equation*}
\big|e^{-t_{\mathcal{B}}\mathcal{L}}f(x)-f_{\mathcal{B}}\big|
\leq\int_{\mathbb R^n}\big|\mathcal{P}_{t_{\mathcal{B}}}(x,y)\big|\cdot\big|f(y)-f_{\mathcal{B}}\big|\,dy.
\end{equation*}
It is clear that for any $x\in \mathcal{B}$ and $y\in 2\mathcal{B}$, we have
\begin{equation*}
\big|\mathcal{P}_{t_{\mathcal{B}}}(x,y)\big|\leq\frac{C}{(t_{\mathcal{B}})^{n/2}}\leq\frac{C}{m(2\mathcal{B})}.
\end{equation*}
Moreover, for any $x\in \mathcal{B}$ and $y\in(2\mathcal{B})^{\complement}$, we have
\begin{equation*}
|x-y|\geq\frac{|y-x_0|}{2}\geq r_{\mathcal{B}}.
\end{equation*}
Consequently, in this case,
\begin{equation*}
\big|\mathcal{P}_{t_{\mathcal{B}}}(x,y)\big|\leq C\cdot\frac{(t_{\mathcal{B}})^{n/2}}{|x-y|^{2n}}
\leq C\cdot\frac{(t_{\mathcal{B}})^{n/2}}{|y-x_0|^{2n}}
\end{equation*}
by using \eqref{G}. We then decompose $\mathbb R^n$ into a geometrically increasing sequence of concentric balls, and obtain
\begin{equation*}
\begin{split}
&\big|e^{-t_{\mathcal{B}}\mathcal{L}}f(x)-f_{\mathcal{B}}\big|\\
&\leq\int_{2\mathcal{B}}\big|\mathcal{P}_{t_{\mathcal{B}}}(x,y)\big|\cdot\big|f(y)-f_{\mathcal{B}}\big|\,dy
+\sum_{k=1}^{\infty}\int_{2^{k+1}\mathcal{B}\setminus 2^k \mathcal{B}}\big|\mathcal{P}_{t_{\mathcal{B}}}(x,y)\big|\cdot\big|f(y)-f_{\mathcal{B}}\big|\,dy\\
&\leq C\frac{1}{m(2\mathcal{B})}\int_{2\mathcal{B}}\big|f(y)-f_{\mathcal{B}}\big|\,dy
+C\sum_{k=1}^{\infty}\int_{2^{k+1}\mathcal{B}\setminus 2^k \mathcal{B}}
\frac{(t_{\mathcal{B}})^{n/2}}{|y-x_0|^{2n}}\cdot\big|f(y)-f_{\mathcal{B}}\big|\,dy\\
&\leq C\frac{1}{m(2\mathcal{B})}\int_{2\mathcal{B}}\big|f(y)-f_{\mathcal{B}}\big|\,dy
+C\sum_{k=1}^{\infty}\frac{1}{2^{kn}}\frac{1}{m(2^{k+1}\mathcal{B})}\int_{2^{k+1}\mathcal{B}}\big|f(y)-f_{\mathcal{B}}\big|\,dy.
\end{split}
\end{equation*}
For any $k\in \mathbb{N}$, it can be verified that (see Lemma \ref{wanglemma1} below)
\begin{equation*}
\frac{1}{m(2^k\mathcal{B})}\int_{2^k\mathcal{B}}\big|f(y)-f_{\mathcal{B}}\big|\,dy
\leq Ck\Big[m\big(2^k\mathcal{B}\big)^{\gamma/n}\cdot\big\|f\big\|_{\mathcal{C}^{p,\gamma}}\Big].
\end{equation*}
Applying the above estimate, we can deduce that for almost all $x\in \mathcal{B}$,
\begin{equation}\label{22}
\begin{split}
&\big|e^{-t_{\mathcal{B}}\mathcal{L}}f(x)-f_{\mathcal{B}}\big|\\
&\leq C\bigg[m(2\mathcal{B})^{\gamma/n}+\sum_{k=1}^{\infty}\frac{1}{2^{kn}}(k+1)m\big(2^{k+1}\mathcal{B}\big)^{\gamma/n}\bigg]
\big\|f\big\|_{\mathcal{C}^{p,\gamma}}\\
&\leq C\bigg[m(\mathcal{B})^{\gamma/n}+\sum_{k=1}^{\infty}\frac{k}{2^{k(n-\gamma)}}m(\mathcal{B})^{\gamma/n}\bigg]
\big\|f\big\|_{\mathcal{C}^{p,\gamma}}\\
&\leq Cm(\mathcal{B})^{\gamma/n}\big\|f\big\|_{\mathcal{C}^{p,\gamma}},
\end{split}
\end{equation}
where the last inequality follows from the fact that $\gamma<n$. Therefore, by using Minkowski's inequality, \eqref{21} and \eqref{22}, we obtain
\begin{equation*}
\begin{split}
&\frac{1}{m(\mathcal{B})^{\gamma/n}}\bigg(\frac{1}{m(\mathcal{B})}
\int_{\mathcal{B}}\big|f(x)-e^{-t_{\mathcal{B}}\mathcal{L}}f(x)\big|^{p}\,dx\bigg)^{1/{p}}\\
&\leq\frac{1}{m(\mathcal{B})^{\gamma/n}}\bigg(\frac{1}{m(\mathcal{B})}
\int_{\mathcal{B}}\big|f(x)-f_{\mathcal{B}}\big|^{p}\,dx\bigg)^{1/{p}}
+\frac{1}{m(\mathcal{B})^{\gamma/n}}\bigg(\frac{1}{m(\mathcal{B})}
\int_{\mathcal{B}}\big|e^{-t_{\mathcal{B}}\mathcal{L}}f(x)-f_{\mathcal{B}}\big|^{p}\,dx\bigg)^{1/{p}}\\
&\leq C\big\|f\big\|_{\mathcal{C}^{p,\gamma}}.
\end{split}
\end{equation*}
Hence, the desired result \eqref{110} follows by taking the supremum over all balls $\mathcal{B}$ in $\mathbb R^n$, and then the proof of Lemma \ref{inclusion} is complete. We also remark that the condition $e^{-t\mathcal{L}}(\mathbf{1})=\mathbf{1}$ is necessary for \eqref{110}.
\end{proof}

In this paper, the symbols $\mathbb R$ and $\mathbb N$ stand for the sets of all real numbers and natural numbers, respectively. $C>0$ denotes a universal constant which is independent of the main parameters involved and may change from line to line. The notation $\mathbf{X}\approx\mathbf{Y}$ means that $C_1\mathbf{Y}\leq \mathbf{X}\leq C_2\mathbf{Y}$ with some positive constants $C_1$ and $C_2$. $B(x_0,r_{B})$ denotes the ball centered at $x_0$ and with radius $r_{B}$. For given $B=B(x_0,r_{B})$ and $\lambda>0$, we write $\lambda B$ for the $\lambda$-dilate ball, which is the ball with the same center $x_0$ and with radius $\lambda r_{B}$. For a measurable set $E$ in $\mathbb R^n$, $m(E)$ denotes the Lebesgue measure of the set $E$ and $\chi_{E}$ denotes the characteristic function of the set $E$.

\section{Some lemmas}
\label{sec2}
In this section, we will give some preliminary lemmas, which will be used in the proof of our main theorems. First, we establish several direct estimates related to the Morrey and Campanato spaces.

\begin{lemma}\label{wanglemma1}
Let $b\in \mathcal{C}^{p_1,\beta_1}(\mathbb R^n)$ with $1\leq p_1<\infty$ and $-n/{p_1}\leq\beta_1<0$. Then the following properties hold.
\begin{enumerate}
  \item[$(1)$] For any ball $\mathcal{B}$ in $\mathbb R^n$, we get
  \begin{equation*}
  \frac{1}{m(\mathcal{B})}\int_{\mathcal{B}}\big|b(x)-b_{\mathcal{B}}\big|\,dx
  \leq m(\mathcal{B})^{\beta_1/n}\cdot\big\|b\big\|_{\mathcal{C}^{p_1,\beta_1}}.
  \end{equation*}
  \item[$(2)$] For every $k\in \mathbb{N}$, we get
\begin{equation*}
\frac{1}{m(2^k\mathcal{B})}\int_{2^k\mathcal{B}}\big|b(x)-b_{\mathcal{B}}\big|\,dx
\leq (2^n+1)k\Big[m\big(2^k\mathcal{B}\big)^{\beta_1/n}\cdot\big\|b\big\|_{\mathcal{C}^{p_1,\beta_1}}\Big].
\end{equation*}
  \item[$(3)$] Let $f\in \mathcal{M}^{p_2,\beta_2}(\mathbb R^n)$ with $1\leq p_2<\infty$ and $-n/{p_2}\leq\beta_2<0$. Then for every $k\in \mathbb{N}$, we get
\begin{equation*}
\frac{1}{m(2^k\mathcal{B})}\int_{2^k\mathcal{B}}\big|b(x)-b_{\mathcal{B}}\big|\cdot\big|f(x)\big|\,dx
\leq (2^n+1)k\Big[m\big(2^{k}\mathcal{B}\big)^{{(\beta_1+\beta_2)}/n}
\cdot\big\|b\big\|_{\mathcal{C}^{p_1,\beta_1}}\big\|f\big\|_{\mathcal{M}^{p_2,\beta_2}}\Big].
\end{equation*}
\end{enumerate}
\end{lemma}

\begin{proof}
We begin with the proof of (1). From H\"{o}lder's inequality and the definition of $\mathcal{C}^{p_1,\beta_1}(\mathbb R^n)$, it then follows that
\begin{equation*}
\begin{split}
\frac{1}{m(\mathcal{B})}\int_{\mathcal{B}}\big|b(x)-b_{\mathcal{B}}\big|\,dx
&\leq\bigg(\frac{1}{m(\mathcal{B})}\int_{\mathcal{B}}\big|b(x)-b_{\mathcal{B}}\big|^{p_1}\,dx\bigg)^{1/{p_1}}\\
&\leq m(\mathcal{B})^{\beta_1/n}\cdot\big\|b\big\|_{\mathcal{C}^{p_1,\beta_1}}.
\end{split}
\end{equation*}
Next, we give the proof of (2). Observe that for any ball $\mathcal{B}$ in $\mathbb R^n$, we have
\begin{equation*}
\begin{split}
\big|b_{2\mathcal{B}}-b_{\mathcal{B}}\big|
&\leq\frac{1}{m(\mathcal{B})}\int_{\mathcal{B}}\big|b(x)-b_{2\mathcal{B}}\big|\,dx\\
&\leq\frac{2^n}{m(2\mathcal{B})}\int_{2\mathcal{B}}\big|b(x)-b_{2\mathcal{B}}\big|\,dx\leq 2^nm(2\mathcal{B})^{\beta_1/n}\cdot\big\|b\big\|_{\mathcal{C}^{p_1,\beta_1}},
\end{split}
\end{equation*}
by part (1). Similarly, for each $1\leq i\leq k$,
\begin{equation*}
\big|b_{2^i\mathcal{B}}-b_{2^{i-1}\mathcal{B}}\big|
\leq 2^nm\big(2^i\mathcal{B}\big)^{\beta_1/n}\cdot\big\|b\big\|_{\mathcal{C}^{p_1,\beta_1}}.
\end{equation*}
Hence, for any $k\in \mathbb{N}$, we can deduce that
\begin{equation*}
\begin{split}
\frac{1}{m(2^k\mathcal{B})}\int_{2^k\mathcal{B}}\big|b(x)-b_{\mathcal{B}}\big|\,dx
&\leq\frac{1}{m(2^k\mathcal{B})}\int_{2^k\mathcal{B}}\big|b(x)-b_{2^k\mathcal{B}}\big|\,dx
+\big|b_{2^k\mathcal{B}}-b_{\mathcal{B}}\big|\\
&\leq m\big(2^k\mathcal{B}\big)^{\beta_1/n}\cdot\big\|b\big\|_{\mathcal{C}^{p_1,\beta_1}}
+\sum_{i=1}^k\big|b_{2^i\mathcal{B}}-b_{2^{i-1}\mathcal{B}}\big|\\
&\leq m\big(2^k\mathcal{B}\big)^{\beta_1/n}\cdot\big\|b\big\|_{\mathcal{C}^{p_1,\beta_1}}
+\sum_{i=1}^k 2^nm\big(2^k\mathcal{B}\big)^{\beta_1/n}\cdot\big\|b\big\|_{\mathcal{C}^{p_1,\beta_1}}\\
&\leq(2^n+1)k\Big[m\big(2^k\mathcal{B}\big)^{\beta_1/n}\cdot\big\|b\big\|_{\mathcal{C}^{p_1,\beta_1}}\Big].
\end{split}
\end{equation*}
Let us now prove (3). We take $1\leq p<\infty$ such that
\begin{equation*}
1/p=1/{p_1}+1/{p_2}.
\end{equation*}
By using H\"{o}lder's inequality and the definitions of $\mathcal{C}^{p_1,\beta_1}(\mathbb R^n)$ and $\mathcal{M}^{p_2,\beta_2}(\mathbb R^n)$, we obtain that for any $k\in \mathbb{N}$,
\begin{equation*}
\begin{split}
&\frac{1}{m(2^{k}\mathcal{B})}\int_{2^{k}\mathcal{B}}\big|b(x)-b_{2^{k}\mathcal{B}}\big|\cdot\big|f(x)\big|\,dx\\
&\leq\bigg(\frac{1}{m(2^{k}\mathcal{B})}\int_{2^{k}\mathcal{B}}\big|b(x)-b_{2^{k}\mathcal{B}}\big|^p\big|f(x)\big|^p\,dx\bigg)^{1/p}\\
&\leq\frac{1}{m(2^{k}\mathcal{B})^{1/p}}\bigg(\int_{2^{k}\mathcal{B}}\big|b(x)-b_{2^{k}\mathcal{B}}\big|^{p_1}dx\bigg)^{1/{p_1}}
\bigg(\int_{2^{k}\mathcal{B}}\big|f(x)\big|^{p_2}dx\bigg)^{1/{p_2}}\\
&\leq\frac{1}{m(2^{k}\mathcal{B})^{1/p}}
\cdot m\big(2^{k}\mathcal{B}\big)^{\beta_1/n+1/{p_1}}\cdot m\big(2^{k}\mathcal{B}\big)^{\beta_2/n+1/{p_2}}
\big\|b\big\|_{\mathcal{C}^{p_1,\beta_1}}\big\|f\big\|_{\mathcal{M}^{p_2,\beta_2}}\\
&=m\big(2^{k}\mathcal{B}\big)^{{(\beta_1+\beta_2)}/n}
\cdot\big\|b\big\|_{\mathcal{C}^{p_1,\beta_1}}\big\|f\big\|_{\mathcal{M}^{p_2,\beta_2}}.
\end{split}
\end{equation*}
Hence, by part (2), we have
\begin{equation*}
\begin{split}
&\frac{1}{m(2^k\mathcal{B})}\int_{2^k\mathcal{B}}\big|b(x)-b_{\mathcal{B}}\big|\cdot\big|f(x)\big|\,dx\\
&\leq\frac{1}{m(2^{k}\mathcal{B})}\int_{2^{k}\mathcal{B}}\big|b(x)-b_{2^{k}\mathcal{B}}\big|\cdot\big|f(x)\big|\,dx
+\big|b_{2^k\mathcal{B}}-b_{\mathcal{B}}\big|\cdot\frac{1}{m(2^k\mathcal{B})}\int_{2^k\mathcal{B}}\big|f(x)\big|\,dx\\
&\leq m\big(2^{k}\mathcal{B}\big)^{{(\beta_1+\beta_2)}/n}
\cdot\big\|b\big\|_{\mathcal{C}^{p_1,\beta_1}}\big\|f\big\|_{\mathcal{M}^{p_2,\beta_2}}
+\sum_{i=1}^k 2^nm\big(2^k\mathcal{B}\big)^{\beta_1/n}\cdot\big\|b\big\|_{\mathcal{C}^{p_1,\beta_1}}
\cdot\frac{1}{m(2^k\mathcal{B})}\int_{2^k\mathcal{B}}\big|f(x)\big|\,dx.
\end{split}
\end{equation*}
Moreover, by H\"{o}lder's inequality again,
\begin{equation*}
\begin{split}
\frac{1}{m(2^k\mathcal{B})}\int_{2^k\mathcal{B}}\big|f(x)\big|\,dx
&\leq\bigg(\frac{1}{m(2^k\mathcal{B})}\int_{2^k\mathcal{B}}\big|f(x)\big|^{p_2}dx\bigg)^{1/{p_2}}\\
&\leq m\big(2^{k}\mathcal{B}\big)^{\beta_2/n}\big\|f\big\|_{\mathcal{M}^{p_2,\beta_2}}.
\end{split}
\end{equation*}
Therefore,
\begin{equation*}
\begin{split}
&\frac{1}{m(2^k\mathcal{B})}\int_{2^k\mathcal{B}}\big|b(x)-b_{\mathcal{B}}\big|\cdot\big|f(x)\big|\,dx\\
&\leq m\big(2^{k}\mathcal{B}\big)^{{(\beta_1+\beta_2)}/n}
\cdot\big\|b\big\|_{\mathcal{C}^{p_1,\beta_1}}\big\|f\big\|_{\mathcal{M}^{p_2,\beta_2}}
+\sum_{i=1}^k 2^n m\big(2^k\mathcal{B}\big)^{{(\beta_1+\beta_2)}/n}
\cdot\big\|b\big\|_{\mathcal{C}^{p_1,\beta_1}}\big\|f\big\|_{\mathcal{M}^{p_2,\beta_2}}\\
&\leq (2^n+1)k\Big[m\big(2^{k}\mathcal{B}\big)^{{(\beta_1+\beta_2)}/n}
\cdot\big\|b\big\|_{\mathcal{C}^{p_1,\beta_1}}\big\|f\big\|_{\mathcal{M}^{p_2,\beta_2}}\Big].
\end{split}
\end{equation*}
We are done.
\end{proof}

We also need the following key lemma, which gives the kernel estimate of the difference operator $\mathcal{L}^{-\alpha/2}-e^{-t\mathcal L}\mathcal{L}^{-\alpha/2}$ for any $t>0$ and $0<\alpha<n$.

\begin{lemma}\label{wanglemma2}
Assume that the semigroup $\big\{e^{-t\mathcal L}\big\}_{t>0}$ has a kernel $\mathcal{P}_t(x,y)$ satisfying the Gaussian upper bound \eqref{G}.
Then for any $0<\alpha<n$, the difference operator $(I-e^{-t\mathcal L})\mathcal{L}^{-\alpha/2}:=\mathcal{L}^{-\alpha/2}-e^{-t\mathcal L}\mathcal{L}^{-\alpha/2}$ has an associated kernel $\widetilde{K}_{\alpha,t}(x,y)$ which satisfies
\begin{equation*}
\widetilde{K}_{\alpha,t}(x,y)\leq\frac{C}{|x-y|^{n-\alpha}}\cdot\frac{t}{|x-y|^2}.
\end{equation*}
\end{lemma}
Here the constant $C$ is independent of $x,y\in\mathbb R^n$ and $t\in(0,+\infty)$. For the proof of this lemma, see Duong--Yan
\cite[Lemma 3.1]{duong1} for $0<\alpha<1$ and Deng--Duong--Sikora--Yan \cite[Lemma 5.3]{deng} for $0<\alpha<n$.

\section{Main results}
\label{sec3}
In this section, we will state and prove the main results of this paper. It is proved that for all $0<\alpha<n$, the commutator
$\big[b,\mathcal{L}^{-\alpha/2}\big]$ is bounded from $\mathcal{M}^{p_2,\beta_2}(\mathbb R^n)$ to $\mathcal{C}^{q,\gamma}_{\mathcal{L}}(\mathbb R^n)$ for appropriate indices, when the symbol function $b(x)$ belongs to the space $\mathcal{C}^{p_1,\beta_1}(\mathbb R^n)$ with $1\leq p_1<\infty$ and $-n/{p_1}\leq\beta_1<0$. The corresponding results for the higher-order commutators are also obtained.

\begin{theorem}\label{thm1}
Let $0<\alpha<n$, $1<p_2<n/{\alpha}$ and $-n/{p_2}\leq\beta_2<(-\alpha)$. Suppose that $b\in \mathcal{C}^{p_1,\beta_1}(\mathbb R^n)$ with $1\leq p_1<\infty$ and $-n/{p_1}\leq\beta_1<0$. Then for any $f\in \mathcal{M}^{p_2,\beta_2}(\mathbb R^n)$, there exists a positive constant $C>0$ independent of $b$ and $f$ such that
\begin{equation*}
\big\|\big[b,\mathcal L^{-\alpha/2}\big](f)\big\|_{\mathcal{C}^{q,\gamma}_{\mathcal{L}}}
\leq C\big\|b\big\|_{\mathcal{C}^{p_1,\beta_1}}\big\|f\big\|_{\mathcal{M}^{p_2,\beta_2}},
\end{equation*}
provided that
\begin{equation*}
1/q=1/{p_1}+1/{p_2}-\alpha/n\quad \& \quad \gamma=\beta_1+\beta_2+\alpha.
\end{equation*}
\end{theorem}

\begin{proof}[Proof of Theorem $\ref{thm1}$]
Let $f\in \mathcal{M}^{p_2,\beta_2}(\mathbb R^n)$ with $1<p_2<n/{\alpha}$ and $-n/{p_2}\leq\beta_2<(-\alpha)$. By the definition of $\mathcal{C}^{q,\gamma}_{\mathcal{L}}(\mathbb R^n)$, it suffices to prove that for any given ball $\mathcal{B}$ in $\mathbb R^n$,
\begin{equation}\label{main41}
\frac{1}{m(\mathcal{B})^{{\gamma}/n}}\bigg(\frac{1}{m(\mathcal{B})}
\int_{\mathcal{B}}\Big|\big[b,\mathcal L^{-\alpha/2}\big](f)(x)-e^{-t_{\mathcal{B}}\mathcal{L}}
\big(\big[b,\mathcal L^{-\alpha/2}\big]f\big)(x)\Big|^qdx\bigg)^{1/{q}}
\leq C\big\|b\big\|_{\mathcal{C}^{p_1,\beta_1}}\big\|f\big\|_{\mathcal{M}^{p_2,\beta_2}}.
\end{equation}
Let $\mathcal{B}=B(x_0,r_{\mathcal{B}})$ be a fixed ball centered at $x_0$ and of radius $r_{\mathcal{B}}$. We decompose the function $f(x)$ as follows:
\begin{equation*}
f(x)=f(x)\cdot\chi_{2\mathcal{B}}+f(x)\cdot\chi_{(2\mathcal{B})^{\complement}}:=f_1(x)+f_2(x),
\end{equation*}
where $2\mathcal{B}=B(x_0,2r_{\mathcal{B}})$ and $(2\mathcal{B})^{\complement}=\mathbb R^n\setminus(2\mathcal{B})$. Observe that
\begin{equation*}
\big[b,\mathcal L^{-\alpha/2}\big](f)(x)=\big[b(x)-b_{2\mathcal{B}}\big]\cdot \mathcal{L}^{-\alpha/2}(f)(x)
-\mathcal{L}^{-\alpha/2}\big(\big[b-b_{2\mathcal{B}}\big]f_1\big)(x)-\mathcal{L}^{-\alpha/2}\big(\big[b-b_{2\mathcal{B}}\big]f_2\big)(x),
\end{equation*}
and
\begin{equation*}
\begin{split}
e^{-t_{\mathcal{B}}\mathcal{L}}\big(\big[b,\mathcal L^{-\alpha/2}\big]f\big)(x)
&=e^{-t_{\mathcal{B}}\mathcal{L}}\big(\big[b-b_{2\mathcal{B}}\big]\cdot\mathcal{L}^{-\alpha/2}(f)\big)(x)\\
&-e^{-t_{\mathcal{B}}\mathcal{L}}\mathcal{L}^{-\alpha/2}\big(\big[b-b_{2\mathcal{B}}\big]f_1\big)(x)-
e^{-t_{\mathcal{B}}\mathcal{L}}\mathcal{L}^{-\alpha/2}\big(\big[b-b_{2\mathcal{B}}\big]f_2\big)(x).
\end{split}
\end{equation*}
Here $t_{\mathcal{B}}=r_{\mathcal{B}}^2$. Then we have
\begin{equation*}
\begin{split}
&\frac{1}{m(\mathcal{B})^{{\gamma}/n}}\bigg(\frac{1}{m(\mathcal{B})}
\int_{\mathcal{B}}\Big|\big[b,\mathcal L^{-\alpha/2}\big](f)(x)-e^{-t_{\mathcal{B}}\mathcal{L}}
\big(\big[b,\mathcal L^{-\alpha/2}\big]f\big)(x)\Big|^qdx\bigg)^{1/{q}}\\
\leq&\frac{1}{m(\mathcal{B})^{{\gamma}/n}}\bigg(\frac{1}{m(\mathcal{B})}
\int_{\mathcal{B}}\Big|\big[b(x)-b_{2\mathcal{B}}\big]\cdot\mathcal{L}^{-\alpha/2}(f)(x)\Big|^{q}dx\bigg)^{1/{q}}
+\frac{1}{m(\mathcal{B})^{{\gamma}/n}}\bigg(\frac{1}{m(\mathcal{B})}
\int_{\mathcal{B}}\Big|\mathcal{L}^{-\alpha/2}\big(\big[b-b_{2\mathcal{B}}\big]f_1\big)(x)\Big|^{q}dx\bigg)^{1/{q}}\\
+&\frac{1}{m(\mathcal{B})^{{\gamma}/n}}\bigg(\frac{1}{m(\mathcal{B})}
\int_{\mathcal{B}}\Big|e^{-t_{\mathcal{B}}\mathcal{L}}\big(\big[b-b_{2\mathcal{B}}\big]
\cdot\mathcal{L}^{-\alpha/2}(f)\big)(x)\Big|^{q}dx\bigg)^{1/{q}}\\
+&\frac{1}{m(\mathcal{B})^{{\gamma}/n}}\bigg(\frac{1}{m(\mathcal{B})}
\int_{\mathcal{B}}\Big|e^{-t_{\mathcal{B}}\mathcal{L}}\mathcal{L}^{-\alpha/2}\big(\big[b-b_{2\mathcal{B}}\big]f_1\big)(x)\Big|^{q}dx\bigg)^{1/{q}}\\
+&\frac{1}{m(\mathcal{B})^{{\gamma}/n}}\bigg(\frac{1}{m(\mathcal{B})}
\int_{\mathcal{B}}\Big|\mathcal{L}^{-\alpha/2}\big(\big[b-b_{2\mathcal{B}}\big]f_2\big)(x)
-e^{-t_{\mathcal{B}}\mathcal{L}}\mathcal{L}^{-\alpha/2}\big(\big[b-b_{2\mathcal{B}}\big]f_2\big)(x)\Big|^{q}dx\bigg)^{1/{q}}\\
:=&\mathrm{I+II+III+IV+V}.
\end{split}
\end{equation*}
Let us now give the estimates of I, II, III, IV and V, respectively. For the first term I, we set
\begin{equation*}
1/s=1/{p_2}-\alpha/n.
\end{equation*}
Then
\begin{equation}\label{main42}
1/q=1/{p_1}+1/s.
\end{equation}
We apply H\"{o}lder's inequality and \eqref{main42} to obtain
\begin{equation*}
\begin{split}
\mathrm{I}&\leq\frac{1}{m(\mathcal{B})^{{\gamma}/n+1/q}}\bigg(\int_{\mathcal{B}}\big|b(x)-b_{2\mathcal{B}}\big|^{p_1}dx\bigg)^{1/{p_1}}
\times\bigg(\int_{\mathcal{B}}\big|\mathcal{L}^{-\alpha/2}(f)(x)\big|^{s}dx\bigg)^{1/{s}}\\
&\leq\big\|b\big\|_{\mathcal{C}^{p_1,\beta_1}}\times
\frac{1}{m(\mathcal{B})^{{\gamma}/n+1/q-\beta_1/n-1/{p_1}}}
\bigg(\int_{\mathcal{B}}\big|\mathcal{L}^{-\alpha/2}(f)(x)\big|^{s}dx\bigg)^{1/{s}}\\
&=\big\|b\big\|_{\mathcal{C}^{p_1,\beta_1}}\times\frac{1}{m(\mathcal{B})^{{(\gamma-\beta_1)}/n}}
\bigg(\frac{1}{m(\mathcal{B})}\int_{\mathcal{B}}\big|\mathcal{L}^{-\alpha/2}(f)(x)\big|^{s}dx\bigg)^{1/{s}}.
\end{split}
\end{equation*}
Note that
\begin{equation}\label{main43}
\gamma-\beta_1=\beta_2+\alpha<0.
\end{equation}
According to Theorem \ref{generalizedmorrey}, we know that $\mathcal{L}^{-\alpha/2}$ is bounded from $\mathcal{M}^{p_2,\beta_2}(\mathbb R^n)$ to $\mathcal{M}^{s,\alpha+\beta_2}(\mathbb R^n)$. This fact, together with \eqref{main43}, gives
\begin{equation*}
\begin{split}
\mathrm{I}\leq\big\|b\big\|_{\mathcal{C}^{p_1,\beta_1}}\big\|\mathcal{L}^{-\alpha/2}(f)\big\|_{\mathcal{M}^{s,\alpha+\beta_2}}
\leq C\big\|b\big\|_{\mathcal{C}^{p_1,\beta_1}}\big\|f\big\|_{\mathcal{M}^{p_2,\beta_2}}.
\end{split}
\end{equation*}
For the second term II, we set
\begin{equation*}
1/p=1/{p_1}+1/{p_2}.
\end{equation*}
Then we have
\begin{equation}\label{main44}
1/q=1/{p}-\alpha/n.
\end{equation}
From H\"{o}lder's inequality and Theorem \ref{12}, it then follows that
\begin{equation*}
\begin{split}
\mathrm{II}
&\leq\frac{C}{m(\mathcal{B})^{{\gamma}/n+1/q}}\bigg(\int_{2\mathcal{B}}
\big|b(x)-b_{2\mathcal{B}}\big|^p\cdot\big|f(x)\big|^pdx\bigg)^{1/p}\\
&\leq\frac{C}{m(\mathcal{B})^{{\gamma}/n+1/q}}
\bigg(\int_{2\mathcal{B}}\big|b(x)-b_{2\mathcal{B}}\big|^{p_1}dx\bigg)^{1/{p_1}}
\bigg(\int_{2\mathcal{B}}\big|f(x)\big|^{p_2}dx\bigg)^{1/{p_2}}.
\end{split}
\end{equation*}
Moreover, by using \eqref{main44} and the definition of $\mathcal{C}^{p_1,\beta_1}(\mathbb R^n)$ and $\mathcal{M}^{p_2,\beta_2}(\mathbb R^n)$, one can see that
\begin{equation*}
\begin{split}
\mathrm{II}&\leq\frac{C}{m(\mathcal{B})^{{\gamma}/n+1/q}}
\cdot m(2\mathcal{B})^{\beta_1/n+1/{p_1}}\cdot m(2\mathcal{B})^{\beta_2/n+1/{p_2}}
\big\|b\big\|_{\mathcal{C}^{p_1,\beta_1}}\big\|f\big\|_{\mathcal{M}^{p_2,\beta_2}}\\
&=\frac{C}{m(\mathcal{B})^{{\gamma}/n+1/q}}\cdot m(2\mathcal{B})^{{(\beta_1+\beta_2+\alpha)}/n+1/{q}}
\big\|b\big\|_{\mathcal{C}^{p_1,\beta_1}}\big\|f\big\|_{\mathcal{M}^{p_2,\beta_2}}\\
&\leq C\big\|b\big\|_{\mathcal{C}^{p_1,\beta_1}}\big\|f\big\|_{\mathcal{M}^{p_2,\beta_2}},
\end{split}
\end{equation*}
where the last inequality follows from the fact that $\gamma=\beta_1+\beta_2+\alpha$. As for the estimate of the term III, we first note that for any $x\in \mathcal{B}$ and $y\in 2\mathcal{B}$, the following estimate
\begin{equation}\label{main45}
\big|\mathcal{P}_{t_{\mathcal{B}}}(x,y)\big|\leq\frac{C}{(t_{\mathcal{B}})^{n/2}}\leq\frac{C}{m(2\mathcal{B})}
\end{equation}
holds by \eqref{G}. Moreover, for any $x\in \mathcal{B}$ and $y\in 2^{k+1}\mathcal{B}\setminus 2^k \mathcal{B}$ with $k\in \mathbb{N}$, one has $|y-x|\approx|y-x_0|$, and hence
\begin{equation}\label{main46}
\big|\mathcal{P}_{t_{\mathcal{B}}}(x,y)\big|\leq C\cdot\frac{(t_{\mathcal{B}})^{n/2}}{|x-y|^{2n}}
\leq C\cdot\frac{(t_{\mathcal{B}})^{n/2}}{|y-x_0|^{2n}}
\end{equation}
by using \eqref{G} again. Therefore, it follows from \eqref{main45} and \eqref{main46} that for any $x\in\mathcal{B}$,
\begin{equation*}
\begin{split}
&\Big|e^{-t_{\mathcal{B}}\mathcal{L}}\big(\big[b-b_{2\mathcal{B}}\big]
\cdot\mathcal{L}^{-\alpha/2}(f)\big)(x)\Big|
=\bigg|\int_{\mathbb R^n}\mathcal{P}_{t_{\mathcal{B}}}(x,y)\cdot\big[b(y)-b_{2\mathcal{B}}\big]\cdot\mathcal{L}^{-\alpha/2}(f)(y)\,dy\bigg|\\
&\leq\int_{2\mathcal{B}}\big|\mathcal{P}_{t_{\mathcal{B}}}(x,y)\big|\cdot\big|b(y)-b_{2\mathcal{B}}\big|
\cdot\big|\mathcal{L}^{-\alpha/2}(f)(y)\big|\,dy
+\sum_{k=1}^{\infty}\int_{2^{k+1}\mathcal{B}\setminus 2^k \mathcal{B}}\big|\mathcal{P}_{t_{\mathcal{B}}}(x,y)\big|\cdot\big|b(y)-b_{2\mathcal{B}}\big|
\cdot\big|\mathcal{L}^{-\alpha/2}(f)(y)\big|\,dy\\
&\leq C\frac{1}{m(2\mathcal{B})}\int_{2\mathcal{B}}\big|b(y)-b_{2\mathcal{B}}\big|\cdot\big|\mathcal{L}^{-\alpha/2}(f)(y)\big|\,dy
+C\sum_{k=1}^{\infty}\int_{2^{k+1}\mathcal{B}\setminus 2^k \mathcal{B}}
\frac{(t_{\mathcal{B}})^{n/2}}{|y-x_0|^{2n}}\cdot\big|b(y)-b_{2\mathcal{B}}\big|\cdot\big|\mathcal{L}^{-\alpha/2}(f)(y)\big|\,dy\\
&\leq C\frac{1}{m(2\mathcal{B})}\int_{2\mathcal{B}}\big|b(y)-b_{2\mathcal{B}}\big|\cdot\big|\mathcal{L}^{-\alpha/2}(f)(y)\big|\,dy
+C\sum_{k=1}^{\infty}\frac{1}{2^{kn}}\frac{1}{m(2^{k+1}\mathcal{B})}\int_{2^{k+1}\mathcal{B}}
\big|b(y)-b_{2\mathcal{B}}\big|\cdot\big|\mathcal{L}^{-\alpha/2}(f)(y)\big|\,dy.
\end{split}
\end{equation*}
As a consequence, we have
\begin{equation*}
\begin{split}
\mathrm{III}&\leq\frac{C}{m(\mathcal{B})^{{\gamma}/n}}\bigg(\frac{1}{m(2\mathcal{B})}\int_{2\mathcal{B}}
\big|b(y)-b_{2\mathcal{B}}\big|\cdot\big|\mathcal{L}^{-\alpha/2}(f)(y)\big|\,dy\bigg)\\
&+\frac{C}{m(\mathcal{B})^{{\gamma}/n}}\bigg(\sum_{k=1}^{\infty}\frac{1}{2^{kn}}\frac{1}{m(2^{k+1}\mathcal{B})}\int_{2^{k+1}\mathcal{B}}
\big|b(y)-b_{2\mathcal{B}}\big|\cdot\big|\mathcal{L}^{-\alpha/2}(f)(y)\big|\,dy\bigg)\\
&:=\mathrm{III}^{(1)}+\mathrm{III}^{(2)}.
\end{split}
\end{equation*}
As in the proof of the term I, we can also show that
\begin{equation*}
\mathrm{III}^{(1)}\leq \frac{C}{m(\mathcal{B})^{{\gamma}/n}}\bigg(\frac{1}{m(2\mathcal{B})}\int_{2\mathcal{B}}
\big|b(y)-b_{2\mathcal{B}}\big|^q\cdot\big|\mathcal{L}^{-\alpha/2}(f)(y)\big|^qdy\bigg)^{1/q}
\leq C\big\|b\big\|_{\mathcal{C}^{p_1,\beta_1}}\big\|f\big\|_{\mathcal{M}^{p_2,\beta_2}}.
\end{equation*}
On the other hand, for any $k\in \mathbb{N}$,
\begin{equation*}
\begin{split}
&\frac{1}{m(2^{k+1}\mathcal{B})}\int_{2^{k+1}\mathcal{B}}
\big|b(y)-b_{2\mathcal{B}}\big|\cdot\big|\mathcal{L}^{-\alpha/2}(f)(y)\big|\,dy\\
\leq&\frac{1}{m(2^{k+1}\mathcal{B})}\int_{2^{k+1}\mathcal{B}}
\big|b(y)-b_{2^{k+1}\mathcal{B}}\big|\cdot\big|\mathcal{L}^{-\alpha/2}(f)(y)\big|\,dy
+\big|b_{2^{k+1}\mathcal{B}}-b_{2\mathcal{B}}\big|\bigg\{\frac{1}{m(2^{k+1}\mathcal{B})}\int_{2^{k+1}\mathcal{B}}
\big|\mathcal{L}^{-\alpha/2}(f)(y)\big|\,dy\bigg\}.
\end{split}
\end{equation*}
Similarly, we can obtain that
\begin{equation*}
\begin{split}
&\frac{1}{m(2^{k+1}\mathcal{B})}\int_{2^{k+1}\mathcal{B}}
\big|b(y)-b_{2^{k+1}\mathcal{B}}\big|\cdot\big|\mathcal{L}^{-\alpha/2}(f)(y)\big|\,dy\\
&\leq\bigg(\frac{1}{m(2^{k+1}\mathcal{B})}\int_{2^{k+1}\mathcal{B}}
\big|b(y)-b_{2^{k+1}\mathcal{B}}\big|^q\cdot\big|\mathcal{L}^{-\alpha/2}(f)(y)\big|^qdy\bigg)^{1/q}\\
&\leq Cm\big(2^{k+1}\mathcal{B}\big)^{\gamma/n}\big\|b\big\|_{\mathcal{C}^{p_1,\beta_1}}\big\|f\big\|_{\mathcal{M}^{p_2,\beta_2}}.
\end{split}
\end{equation*}
In addition, by Lemma \ref{wanglemma1} and H\"{o}lder's inequality,
\begin{equation*}
\begin{split}
&\big|b_{2^{k+1}\mathcal{B}}-b_{2\mathcal{B}}\big|\bigg\{\frac{1}{m(2^{k+1}\mathcal{B})}\int_{2^{k+1}\mathcal{B}}
\big|\mathcal{L}^{-\alpha/2}(f)(y)\big|\,dy\bigg\}\\
&\leq k\cdot2^nm\big(2^{k+1}\mathcal{B}\big)^{\beta_1/n}\big\|b\big\|_{\mathcal{C}^{p_1,\beta_1}}
\bigg\{\frac{1}{m(2^{k+1}\mathcal{B})}\int_{2^{k+1}\mathcal{B}}
\big|\mathcal{L}^{-\alpha/2}(f)(y)\big|^s\,dy\bigg\}^{1/s}\\
&=k\cdot2^nm\big(2^{k+1}\mathcal{B}\big)^{\gamma/n}\big\|b\big\|_{\mathcal{C}^{p_1,\beta_1}}
\frac{1}{m(2^{k+1}\mathcal{B})^{{(\gamma-\beta_1)}/n}}\bigg\{\frac{1}{m(2^{k+1}\mathcal{B})}\int_{2^{k+1}\mathcal{B}}
\big|\mathcal{L}^{-\alpha/2}(f)(y)\big|^s\,dy\bigg\}^{1/s},
\end{split}
\end{equation*}
where the number $s$ is the same as above. Furthermore, from \eqref{main43} and Theorem \ref{generalizedmorrey}, it then follows that
\begin{equation*}
\begin{split}
&\big|b_{2^{k+1}\mathcal{B}}-b_{2\mathcal{B}}\big|\bigg\{\frac{1}{m(2^{k+1}\mathcal{B})}\int_{2^{k+1}\mathcal{B}}
\big|\mathcal{L}^{-\alpha/2}(f)(y)\big|\,dy\bigg\}\\
&\leq Ck\cdot m\big(2^{k+1}\mathcal{B}\big)^{\gamma/n}\big\|b\big\|_{\mathcal{C}^{p_1,\beta_1}}
\big\|\mathcal{L}^{-\alpha/2}(f)\big\|_{\mathcal{M}^{s,\alpha+\beta_2}}
\leq Ck\cdot m\big(2^{k+1}\mathcal{B}\big)^{\gamma/n}
\big\|b\big\|_{\mathcal{C}^{p_1,\beta_1}}\big\|f\big\|_{\mathcal{M}^{p_2,\beta_2}}.
\end{split}
\end{equation*}
Summing up the above estimates, we conclude that for each $k\in \mathbb{N}$,
\begin{equation*}
\frac{1}{m(2^{k+1}\mathcal{B})}\int_{2^{k+1}\mathcal{B}}
\big|b(y)-b_{2\mathcal{B}}\big|\cdot\big|\mathcal{L}^{-\alpha/2}(f)(y)\big|\,dy
\leq Ck\cdot m\big(2^{k+1}\mathcal{B}\big)^{\gamma/n}
\big\|b\big\|_{\mathcal{C}^{p_1,\beta_1}}\big\|f\big\|_{\mathcal{M}^{p_2,\beta_2}}.
\end{equation*}
This implies that
\begin{equation*}
\begin{split}
\mathrm{III}^{(2)}&\leq C\sum_{k=1}^{\infty}\frac{k}{2^{kn}}
\left[\frac{m(2^{k+1}\mathcal{B})}{m(\mathcal{B})}\right]^{{\gamma}/n}
\big\|b\big\|_{\mathcal{C}^{p_1,\beta_1}}\big\|f\big\|_{\mathcal{M}^{p_2,\beta_2}}\\
&\leq C\sum_{k=1}^{\infty}\frac{k}{2^{k(n-\gamma)}}
\big\|b\big\|_{\mathcal{C}^{p_1,\beta_1}}\big\|f\big\|_{\mathcal{M}^{p_2,\beta_2}}
\leq C\big\|b\big\|_{\mathcal{C}^{p_1,\beta_1}}\big\|f\big\|_{\mathcal{M}^{p_2,\beta_2}},
\end{split}
\end{equation*}
where in the last inequality we have used the fact that $n>0>\gamma$. Consequently,
\begin{equation*}
\mathrm{III} \leq C\big\|b\big\|_{\mathcal{C}^{p_1,\beta_1}}\big\|f\big\|_{\mathcal{M}^{p_2,\beta_2}}.
\end{equation*}
Let us now deal with the term IV. Using the estimates \eqref{main45} and \eqref{main46}, we can deduce that for any $x\in \mathcal{B}$,
\begin{equation*}
\begin{split}
&\Big|e^{-t_{\mathcal{B}}\mathcal{L}}\mathcal{L}^{-\alpha/2}\big(\big[b-b_{2\mathcal{B}}\big]f_1\big)(x)\Big|
=\bigg|\int_{\mathbb R^n}\mathcal{P}_{t_{\mathcal{B}}}(x,y)
\cdot\mathcal{L}^{-\alpha/2}\big(\big[b-b_{2\mathcal{B}}\big]f_1\big)(y)\,dy\bigg|\\
&\leq C\frac{1}{m(2\mathcal{B})}\int_{2\mathcal{B}}\Big|\mathcal{L}^{-\alpha/2}\big(\big[b-b_{2\mathcal{B}}\big]f_1\big)(y)\Big|\,dy
+C\sum_{k=1}^{\infty}\frac{1}{2^{kn}}\frac{1}{m(2^{k+1}\mathcal{B})}\int_{2^{k+1}\mathcal{B}}
\Big|\mathcal{L}^{-\alpha/2}\big(\big[b-b_{2\mathcal{B}}\big]f_1\big)(y)\Big|\,dy.
\end{split}
\end{equation*}
As a consequence, we have
\begin{equation*}
\begin{split}
\mathrm{IV}&\leq\frac{C}{m(\mathcal{B})^{{\gamma}/n}}\bigg(\frac{1}{m(2\mathcal{B})}\int_{2\mathcal{B}}
\Big|\mathcal{L}^{-\alpha/2}\big(\big[b-b_{2\mathcal{B}}\big]f_1\big)(y)\Big|\,dy\bigg)\\
&+\frac{C}{m(\mathcal{B})^{{\gamma}/n}}\bigg(\sum_{k=1}^{\infty}\frac{1}{2^{kn}}\frac{1}{m(2^{k+1}\mathcal{B})}\int_{2^{k+1}\mathcal{B}}
\Big|\mathcal{L}^{-\alpha/2}\big(\big[b-b_{2\mathcal{B}}\big]f_1\big)(y)\Big|\,dy\bigg)\\
&:=\mathrm{IV}^{(1)}+\mathrm{IV}^{(2)}.
\end{split}
\end{equation*}
As in the proof of the term II, we can also prove that
\begin{equation*}
\mathrm{IV}^{(1)}\leq \frac{C}{m(\mathcal{B})^{{\gamma}/n}}\bigg(\frac{1}{m(2\mathcal{B})}\int_{2\mathcal{B}}
\Big|\mathcal{L}^{-\alpha/2}\big(\big[b-b_{2\mathcal{B}}\big]f_1\big)(y)\Big|^qdy\bigg)^{1/q}
\leq C\big\|b\big\|_{\mathcal{C}^{p_1,\beta_1}}\big\|f\big\|_{\mathcal{M}^{p_2,\beta_2}}.
\end{equation*}
On the other hand, Theorem \ref{12} and H\"{o}lder's inequality imply that for any $k\in \mathbb{N}$,
\begin{equation*}
\begin{split}
&\frac{1}{m(2^{k+1}\mathcal{B})}\int_{2^{k+1}\mathcal{B}}
\Big|\mathcal{L}^{-\alpha/2}\big(\big[b-b_{2\mathcal{B}}\big]f_1\big)(y)\Big|\,dy\\
&\leq\bigg(\frac{1}{m(2^{k+1}\mathcal{B})}\int_{2^{k+1}\mathcal{B}}
\Big|\mathcal{L}^{-\alpha/2}\big(\big[b-b_{2\mathcal{B}}\big]f_1\big)(y)\Big|^qdy\bigg)^{1/q}\\
&\leq\frac{1}{m(2^{k+1}\mathcal{B})^{1/q}}
\bigg(\int_{2\mathcal{B}}\big|b(y)-b_{2\mathcal{B}}\big|^p\cdot\big|f(y)\big|^pdy\bigg)^{1/p},
\end{split}
\end{equation*}
where the number $p$ is the same as above. Since $1/p=1/{p_1}+1/{p_2}$, another application of the H\"{o}lder inequality yields
\begin{equation*}
\begin{split}
&\frac{1}{m(2^{k+1}\mathcal{B})}\int_{2^{k+1}\mathcal{B}}
\Big|\mathcal{L}^{-\alpha/2}\big(\big[b-b_{2\mathcal{B}}\big]f_1\big)(y)\Big|\,dy\\
&\leq\frac{1}{m(2^{k+1}\mathcal{B})^{1/q}}\bigg(\int_{2\mathcal{B}}\big|b(y)-b_{2\mathcal{B}}\big|^{p_1}dy\bigg)^{1/{p_1}}
\bigg(\int_{2\mathcal{B}}\big|f(y)\big|^{p_2}dy\bigg)^{1/{p_2}}\\
&\leq\frac{1}{m(2^{k+1}\mathcal{B})^{1/q}}\cdot m(2\mathcal{B})^{{(\beta_1+\beta_2)}/n+1/{p}}
\big\|b\big\|_{\mathcal{C}^{p_1,\beta_1}}\big\|f\big\|_{\mathcal{M}^{p_2,\beta_2}}.
\end{split}
\end{equation*}
This, together with \eqref{main44}, gives us that
\begin{equation*}
\begin{split}
\mathrm{IV}^{(2)}&\leq C\cdot\frac{m(2\mathcal{B})^{\gamma/n}}{m(\mathcal{B})^{{\gamma}/n}}
\sum_{k=1}^{\infty}\frac{1}{2^{kn}}\left[\frac{m(2\mathcal{B})}{m(2^{k+1}\mathcal{B})}\right]^{1/q}
\big\|b\big\|_{\mathcal{C}^{p_1,\beta_1}}\big\|f\big\|_{\mathcal{M}^{p_2,\beta_2}}\\
&\leq C\sum_{k=1}^{\infty}\frac{1}{2^{kn}}
\big\|b\big\|_{\mathcal{C}^{p_1,\beta_1}}\big\|f\big\|_{\mathcal{M}^{p_2,\beta_2}}
\leq C\big\|b\big\|_{\mathcal{C}^{p_1,\beta_1}}\big\|f\big\|_{\mathcal{M}^{p_2,\beta_2}}.
\end{split}
\end{equation*}
Consequently,
\begin{equation*}
\mathrm{IV}\leq C\big\|b\big\|_{\mathcal{C}^{p_1,\beta_1}}\big\|f\big\|_{\mathcal{M}^{p_2,\beta_2}}.
\end{equation*}
It remains to estimate the last term V. Note that if $x\in\mathcal{B}$ and $y\in 2^{k+1}\mathcal{B}\setminus 2^k \mathcal{B}$ with $k\in \mathbb{N}$, then $|y-x|\approx|y-x_0|$. This fact, together with Lemma \ref{wanglemma2}, implies that for any $x\in \mathcal{B}$,
\begin{equation*}
\begin{split}
&\Big|\mathcal{L}^{-\alpha/2}\big(\big[b-b_{2\mathcal{B}}\big]f_2\big)(x)
-e^{-t_{\mathcal{B}}\mathcal{L}}\mathcal{L}^{-\alpha/2}\big(\big[b-b_{2\mathcal{B}}\big]f_2\big)(x)\Big|\\
&=\Big|\big(I-e^{-t_{\mathcal{B}}\mathcal{L}}\big)\mathcal{L}^{-\alpha/2}\big(\big[b-b_{2\mathcal{B}}\big]f_2\big)(x)\Big|\\
&\leq\int_{(2\mathcal{B})^{\complement}}\big|\widetilde{K}_{\alpha,t_{\mathcal{B}}}(x,y)\big|
\cdot\big|\big[b(y)-b_{2\mathcal{B}}\big]f(y)\big|\,dy\\
&\leq C\sum_{k=1}^{\infty}\int_{2^{k+1}\mathcal{B}\setminus 2^k \mathcal{B}}
\frac{1}{|x-y|^{n-\alpha}}\cdot\frac{r_{\mathcal{B}}^2}{|x-y|^{2}}\big|b(y)-b_{2\mathcal{B}}\big|\big|f(y)\big|\,dy\\
&\leq C\sum_{k=1}^{\infty}\frac{1}{2^{2k}}\cdot\frac{1}{m(2^{k+1}\mathcal{B})^{1-\alpha/n}}
\int_{2^{k+1}\mathcal{B}}\big|b(y)-b_{2\mathcal{B}}\big|\big|f(y)\big|\,dy.
\end{split}
\end{equation*}
Hence, by using Lemma \ref{wanglemma1}, we obtain
\begin{equation*}
\begin{split}
\mathrm{V}&=\frac{1}{m(\mathcal{B})^{{\gamma}/n}}\bigg(\frac{1}{m(\mathcal{B})}
\int_{\mathcal{B}}\Big|\mathcal{L}^{-\alpha/2}\big(\big[b-b_{2\mathcal{B}}\big]f_2\big)(x)
-e^{-t_{\mathcal{B}}\mathcal{L}}\mathcal{L}^{-\alpha/2}\big(\big[b-b_{2\mathcal{B}}\big]f_2\big)(x)\Big|^{q}dx\bigg)^{1/{q}}\\
&\leq \frac{C}{m(\mathcal{B})^{{\gamma}/n}}\sum_{k=1}^{\infty}\frac{1}{2^{2k}}\cdot\frac{1}{m(2^{k+1}\mathcal{B})^{1-\alpha/n}}
\int_{2^{k+1}\mathcal{B}}\big|b(y)-b_{2\mathcal{B}}\big|\big|f(y)\big|\,dy\\
&\leq \frac{C}{m(\mathcal{B})^{{\gamma}/n}}
\sum_{k=1}^{\infty}\frac{k}{2^{2k}}\cdot m\big(2^{k+1}\mathcal{B}\big)^{{(\beta_1+\beta_2+\alpha)}/n}
\big\|b\big\|_{\mathcal{C}^{p_1,\beta_1}}\big\|f\big\|_{\mathcal{M}^{p_2,\beta_2}}.\\
\end{split}
\end{equation*}
Since $\gamma=\beta_1+\beta_2+\alpha$, it then follows that
\begin{equation*}
\begin{split}
\mathrm{V}&\leq C\sum_{k=1}^{\infty}\frac{k}{2^{2k}}\cdot\left[\frac{m(2^{k+1}\mathcal{B})}{m(\mathcal{B})}\right]^{{\gamma}/n}
\big\|b\big\|_{\mathcal{C}^{p_1,\beta_1}}\big\|f\big\|_{\mathcal{M}^{p_2,\beta_2}}\\
&\leq C\sum_{k=1}^{\infty}\frac{k}{2^{(2-\gamma)k}}
\big\|b\big\|_{\mathcal{C}^{p_1,\beta_1}}\big\|f\big\|_{\mathcal{M}^{p_2,\beta_2}}\\
&\leq C\big\|b\big\|_{\mathcal{C}^{p_1,\beta_1}}\big\|f\big\|_{\mathcal{M}^{p_2,\beta_2}},
\end{split}
\end{equation*}
where in the last inequality we have used the fact that $\gamma<2$. Combining the above estimates I, II, III, IV with V, we obtain the desired result \eqref{main41}, and hence the proof of Theorem \ref{thm1} is complete.
\end{proof}

In particular, if $\beta_2=-n/{p_2}$, then we can get the corresponding result for $\big[b,\mathcal L^{-\alpha/2}\big]$.
\begin{corollary}
Let $0<\alpha<n$. Suppose that $b\in \mathcal{C}^{p_1,\beta_1}(\mathbb R^n)$ with $1\leq p_1<\infty$ and $-n/{p_1}\leq\beta_1<0$. Then for any
$f\in L^{p_2}(\mathbb R^n)$ with $1<p_2<n/{\alpha}$, there exists a positive constant $C>0$ independent of $b$ and $f$ such that
\begin{equation*}
\big\|\big[b,\mathcal L^{-\alpha/2}\big](f)\big\|_{\mathcal{C}^{q,\gamma}_{\mathcal{L}}}
\leq C\big\|b\big\|_{\mathcal{C}^{p_1,\beta_1}}\big\|f\big\|_{L^{p_2}},
\end{equation*}
provided that
\begin{equation*}
1/q=1/{p_1}+1/{p_2}-\alpha/n\quad \& \quad \gamma=\beta_1+\alpha-n/{p_2}.
\end{equation*}
\end{corollary}
In the present situation, we have
\begin{equation*}
1/q=1/{p_1}+{(\beta_1-\gamma)}/n\quad \& \quad\gamma=\beta_1+n(1/{p_1}-1/q).
\end{equation*}

Let us now study the corresponding estimates for the higher-order commutators. Let $2\leq m\in \mathbb{N}$ and $b(x)$ be a locally integrable function on $\mathbb R^n$. The higher-order commutator $\big[b,\mathcal L^{-\alpha/2}\big]^m$ generated by $\mathcal{L}^{-\alpha/2}$ and $b$ is defined by
\begin{equation*}
\big[b,\mathcal L^{-\alpha/2}\big]^m(f)(x):=\big[b,\dots\big[b,\big[b,\mathcal L^{-\alpha/2}\big]\big]\big](f)(x).
\end{equation*}
Since for $0<\alpha<n$, the kernel of $\mathcal{L}^{-\alpha/2}$ is $\mathcal{K}_{\alpha}(x,y)$, we then have
\begin{equation}\label{last1}
\big[b,\mathcal L^{-\alpha/2}\big]^m(f)(x)=\int_{\mathbb R^n}\big[b(x)-b(y)\big]^m\mathcal{K}_{\alpha}(x,y)f(y)\,dy.
\end{equation}
\begin{itemize}
  \item When $m\geq2$ and $b\in \mathrm{BMO}(\mathbb R^n)$, it was first proved by Mo and Lu that for any $0<\alpha<1$, the higher-order commutator $\big[b,\mathcal{L}^{-\alpha/2}\big]^m$ is bounded from $L^{p}(\mathbb R^n)$ to $L^{q}(\mathbb R^n)$, where $1<p<n/{\alpha}$ and $1/q=1/p-\alpha/n$, see \cite[Theorem 1.7]{mo}. Actually, for all $0<\alpha<n$, the conclusion of Theorem 1.7 in \cite{mo} is also true, by using Lemma \ref{wanglemma2} with the range of $0<\alpha<n$. See also \cite[Theorem 1.4]{auscher} for the weighted case. Furthermore, for all $0<\alpha<n$, it can be shown that the higher-order commutator $\big[b,\mathcal{L}^{-\alpha/2}\big]^m$ is bounded from $\mathcal{M}^{p,\beta}(\mathbb R^n)$ to $\mathcal{M}^{q,\alpha+\beta}(\mathbb R^n)$ by induction on $m$, where $-n/p\leq\beta<-\alpha$. Here we omit the details for brevity.
  \item Let us see what happens if the symbol function $b(x)$ belongs to the space $\mathrm{Lip}_{\beta_1}(\mathbb R^n)$. When $m\geq2$ and $b\in \mathrm{Lip}_{\beta_1}(\mathbb R^n)$ with $0<\beta_1\leq1$ and $0<\alpha+m\beta_1<n$, then by \eqref{last1}, the pointwise inequality \eqref{kernelk} and the definition of $\mathrm{Lip}_{\beta_1}(\mathbb R^n)$, we can deduce that for any $x\in\mathbb R^n$,
\begin{equation}\label{point5}
\begin{split}
\Big|\big[b,\mathcal L^{-\alpha/2}\big]^m(f)(x)\Big|
&\leq\int_{\mathbb R^n}\big|b(x)-b(y)\big|^m\cdot|\mathcal{K}_{\alpha}(x,y)||f(y)|\,dy\\
&\leq C\big\|b\big\|^{m}_{\mathrm{Lip}_{\beta_1}}\int_{\mathbb R^n}\frac{|f(y)|}{|x-y|^{n-\alpha-m\beta_1}}dy\\
&\leq C\big\|b\big\|^{m}_{\mathrm{Lip}_{\beta_1}}I_{\alpha+m\beta_1}(|f|)(x).
\end{split}
\end{equation}
Thus, by \eqref{point5} and Theorem \ref{11}, we can prove that the higher-order commutator $\big[b,\mathcal L^{-\alpha/2}\big]^m$ is a bounded operator from $L^{p}(\mathbb R^n)$ to $L^q(\mathbb R^n)$, whenever
\begin{equation*}
1<p<\frac{n}{\alpha+m\beta_1}\quad \mbox{and} \quad \frac{\,1\,}{q}=\frac{\,1\,}{p}-\frac{\alpha+m\beta_1}{n}.
\end{equation*}
This result was obtained by Mo and Lu in \cite[Theorem 1.8]{mo}. Moreover, by using Theorem \ref{13} and \eqref{point5}, we can also show that the higher-order commutator $\big[b,\mathcal{L}^{-\alpha/2}\big]^m$ is a bounded operator from $\mathcal{M}^{p,\beta}(\mathbb R^n)$ to $\mathcal{M}^{q,\gamma}(\mathbb R^n)$, whenever
\begin{equation*}
-\frac{n}{\,p\,}\leq\beta<-(\alpha+m\beta_1),~~~\frac{\,1\,}{q}=\frac{\,1\,}{p}-\frac{\alpha+m\beta_1}{n},~~~\mbox{and}~~~ \gamma=\alpha+m\beta_1+\beta.
\end{equation*}
\end{itemize}
Inspired by the above results, it will be interesting to consider the boundedness of $\big[b,\mathcal{L}^{-\alpha/2}\big]^m$, when the symbol function $b(x)$ belongs to the space $\mathcal{C}^{p_1,\beta_1}(\mathbb R^n)$ with $1\leq p_1<\infty$ and $-n/{p_1}\leq\beta_1<0$. Below we study the case $m=2$, the general case follows by using the same method.

\begin{theorem}\label{thm2}
Let $0<\alpha<n$, $1<p_2<n/{\alpha}$ and $-n/{p_2}\leq\beta_2<(-\alpha)$. Suppose that $b\in \mathcal{C}^{p_1,\beta_1}(\mathbb R^n)$ with $1\leq p_1<\infty$ and $-n/{p_1}\leq\beta_1<0$. Then for any $f\in \mathcal{M}^{p_2,\beta_2}(\mathbb R^n)$, there exists a positive constant $C>0$ independent of $b$ and $f$ such that
\begin{equation*}
\big\|\big[b,\mathcal{L}^{-\alpha/2}\big]^2(f)\big\|_{\mathcal{C}^{q,\gamma}_{\mathcal{L}}}
\leq C\big\|b\big\|^2_{\mathcal{C}^{p_1,\beta_1}}\big\|f\big\|_{\mathcal{M}^{p_2,\beta_2}},
\end{equation*}
provided that
\begin{equation*}
1/q=2/{p_1}+1/{p_2}-\alpha/n\quad \& \quad \gamma=2\beta_1+\beta_2+\alpha.
\end{equation*}
\end{theorem}

Before proving our main theorem, let us first establish the following result.

\begin{theorem}\label{thm4}
Let $0<\alpha<n$, $1<p_2<n/{\alpha}$ and $-n/{p_2}\leq\beta_2<(-\alpha)$. Suppose that $b\in \mathcal{C}^{p_1,\beta_1}(\mathbb R^n)$ with $1\leq p_1<\infty$ and $-n/{p_1}\leq\beta_1<0$. Then for any $f\in \mathcal{M}^{p_2,\beta_2}(\mathbb R^n)$, there exists a positive constant $C>0$ independent of $b$ and $f$ such that
\begin{equation*}
\big\|\big[b,\mathcal L^{-\alpha/2}\big](f)\big\|_{\mathcal{M}^{q,\gamma}}
\leq C\big\|b\big\|_{\mathcal{C}^{p_1,\beta_1}}\big\|f\big\|_{\mathcal{M}^{p_2,\beta_2}},
\end{equation*}
provided that
\begin{equation*}
1/q=1/{p_1}+1/{p_2}-\alpha/n\quad \& \quad \gamma=\beta_1+\beta_2+\alpha.
\end{equation*}
\end{theorem}

\begin{proof}[Proof of Theorem $\ref{thm4}$]
Let $f\in \mathcal{M}^{p_2,\beta_2}(\mathbb R^n)$ with $1<p_2<n/{\alpha}$ and $-n/{p_2}\leq\beta_2<(-\alpha)$. For any fixed ball $\mathcal{B}=B(x_0,r_{\mathcal{B}})\subset\mathbb R^n$, we decompose the function $f(x)$ into two parts.
\begin{equation*}
f(x)=f(x)\cdot\chi_{2\mathcal{B}}+f(x)\cdot\chi_{(2\mathcal{B})^{\complement}}:=f_1(x)+f_2(x).
\end{equation*}
Observe that for any fixed ball $\mathcal{B}\subset\mathbb R^n$,
\begin{equation*}
\big[b,\mathcal L^{-\alpha/2}\big](f)(x)=\big[b(x)-b_{2\mathcal{B}}\big]\cdot \mathcal{L}^{-\alpha/2}(f)(x)
-\mathcal{L}^{-\alpha/2}\big(\big[b-b_{2\mathcal{B}}\big]f_1\big)(x)-\mathcal{L}^{-\alpha/2}\big(\big[b-b_{2\mathcal{B}}\big]f_2\big)(x).
\end{equation*}
Then we can write
\begin{equation*}
\begin{split}
&\frac{1}{m(\mathcal{B})^{{\gamma}/n}}\bigg(\frac{1}{m(\mathcal{B})}
\int_{\mathcal{B}}\Big|\big[b,\mathcal L^{-\alpha/2}\big](f)(x)\Big|^qdx\bigg)^{1/{q}}\\
\leq&\frac{1}{m(\mathcal{B})^{{\gamma}/n}}\bigg(\frac{1}{m(\mathcal{B})}
\int_{\mathcal{B}}\Big|\big[b(x)-b_{2\mathcal{B}}\big]\cdot\mathcal{L}^{-\alpha/2}(f)(x)\Big|^{q}dx\bigg)^{1/{q}}
+\frac{1}{m(\mathcal{B})^{{\gamma}/n}}\bigg(\frac{1}{m(\mathcal{B})}
\int_{\mathcal{B}}\Big|\mathcal{L}^{-\alpha/2}\big(\big[b-b_{2\mathcal{B}}\big]f_1\big)(x)\Big|^{q}dx\bigg)^{1/{q}}\\
+&\frac{1}{m(\mathcal{B})^{{\gamma}/n}}\bigg(\frac{1}{m(\mathcal{B})}
\int_{\mathcal{B}}\Big|\mathcal{L}^{-\alpha/2}\big(\big[b-b_{2\mathcal{B}}\big]f_2\big)(x)\Big|^{q}dx\bigg)^{1/{q}}\\
:=&\mathrm{I+II+III}.
\end{split}
\end{equation*}
Arguing as in the proof of Theorem \ref{thm1}, we have
\begin{equation*}
\mathrm{I+II}\leq C\big\|b\big\|_{\mathcal{C}^{p_1,\beta_1}}\big\|f\big\|_{\mathcal{M}^{p_2,\beta_2}}.
\end{equation*}
Let us now estimate the term III. Note that if $x\in\mathcal{B}$ and $y\in 2^{k+1}\mathcal{B}\setminus 2^k \mathcal{B}$ with $k\in \mathbb{N}$, then $|y-x|\approx|y-x_0|$. Since for $0<\alpha<n$, the kernel of $\mathcal{L}^{-\alpha/2}$ is $\mathcal{K}_{\alpha}(x,y)$, by using the kernel estimate \eqref{kernelk}, we can deduce that for any $x\in \mathcal{B}$,
\begin{equation*}
\begin{split}
&\Big|\mathcal{L}^{-\alpha/2}\big(\big[b-b_{2\mathcal{B}}\big]f_2\big)(x)\Big|\\
&\leq\int_{(2\mathcal{B})^{\complement}}\big|\mathcal{K}_{\alpha}(x,y)\big|
\cdot\big|\big[b(y)-b_{2\mathcal{B}}\big]f(y)\big|\,dy\\
&\leq C\sum_{k=1}^{\infty}\int_{2^{k+1}\mathcal{B}\setminus 2^k \mathcal{B}}
\frac{1}{|x-y|^{n-\alpha}}\cdot\big|b(y)-b_{2\mathcal{B}}\big|\big|f(y)\big|\,dy\\
&\leq C\sum_{k=1}^{\infty}\frac{1}{m(2^{k+1}\mathcal{B})^{1-\alpha/n}}
\int_{2^{k+1}\mathcal{B}}\big|b(y)-b_{2\mathcal{B}}\big|\big|f(y)\big|\,dy,
\end{split}
\end{equation*}
which together with Lemma \ref{wanglemma1} implies that
\begin{equation*}
\begin{split}
\mathrm{III}&\leq\frac{C}{m(\mathcal{B})^{{\gamma}/n}}\sum_{k=1}^{\infty}\frac{1}{m(2^{k+1}\mathcal{B})^{1-\alpha/n}}
\int_{2^{k+1}\mathcal{B}}\big|b(y)-b_{2\mathcal{B}}\big|\big|f(y)\big|\,dy\\
&\leq\frac{C}{m(\mathcal{B})^{{\gamma}/n}}\sum_{k=1}^{\infty}k\Big[m\big(2^{k+1}\mathcal{B}\big)^{{(\beta_1+\beta_2+\alpha)}/n}
\cdot\big\|b\big\|_{\mathcal{C}^{p_1,\beta_1}}\big\|f\big\|_{\mathcal{M}^{p_2,\beta_2}}\Big].
\end{split}
\end{equation*}
Since $\beta_1<0$ and $\beta_2+\alpha<0$, we see that
\begin{equation*}
\gamma=\beta_1+\beta_2+\alpha<0.
\end{equation*}
Consequently,
\begin{equation*}
\begin{split}
\mathrm{III}&\leq C\sum_{k=1}^{\infty}k\cdot\left[\frac{m(2^{k+1}\mathcal{B})}{m(\mathcal{B})}\right]^{{\gamma}/n}
\big\|b\big\|_{\mathcal{C}^{p_1,\beta_1}}\big\|f\big\|_{\mathcal{M}^{p_2,\beta_2}}\\
&\leq C\sum_{k=1}^{\infty}\frac{k}{(2^{-\gamma})^{k}}
\big\|b\big\|_{\mathcal{C}^{p_1,\beta_1}}\big\|f\big\|_{\mathcal{M}^{p_2,\beta_2}}\\
&\leq C\big\|b\big\|_{\mathcal{C}^{p_1,\beta_1}}\big\|f\big\|_{\mathcal{M}^{p_2,\beta_2}}.
\end{split}
\end{equation*}
Summing up the above estimates of I, II and III, and then taking the supremum over all balls $\mathcal{B}$ in $\mathbb R^n$, we conclude the proof of Theorem \ref{thm4}.
\end{proof}

We are now ready to show our main theorem of this section.

\begin{proof}[Proof of Theorem $\ref{thm2}$]
Let $f\in \mathcal{M}^{p_2,\beta_2}(\mathbb R^n)$ with $1<p_2<n/{\alpha}$ and $-n/{p_2}\leq\beta_2<(-\alpha)$.
For any fixed ball $\mathcal{B}$ in $\mathbb R^n$, it suffices to verify that
\begin{equation}\label{main411}
\begin{split}
\mathbf{W}&:=\frac{1}{m(\mathcal{B})^{{\gamma}/n}}\bigg(\frac{1}{m(\mathcal{B})}
\int_{\mathcal{B}}\Big|\big[b,\mathcal L^{-\alpha/2}\big]^2(f)(x)-e^{-t_{\mathcal{B}}\mathcal{L}}
\big(\big[b,\mathcal L^{-\alpha/2}\big]^2f\big)(x)\Big|^qdx\bigg)^{1/{q}}\\
&\leq C\big\|b\big\|^2_{\mathcal{C}^{p_1,\beta_1}}\big\|f\big\|_{\mathcal{M}^{p_2,\beta_2}}.
\end{split}
\end{equation}
To prove \eqref{main411}, we first observe that for any fixed ball $\mathcal{B}$ in $\mathbb R^n$,
\begin{equation*}
\big[b(x)-b(y)\big]^2=\big[b(x)-b_{2\mathcal{B}}\big]^2-2\big[b(x)-b_{2\mathcal{B}}\big]\cdot\big[b(y)-b_{2\mathcal{B}}\big]
+\big[b(y)-b_{2\mathcal{B}}\big]^2.
\end{equation*}
Let $\mathcal{B}=B(x_0,r_{\mathcal{B}})$ be a fixed ball centered at $x_0$ and with radius $r_{\mathcal{B}}$. As usual, we decompose the function $f(x)$ in the following way
\begin{equation*}
f(x)=f(x)\cdot\chi_{2\mathcal{B}}+f(x)\cdot\chi_{(2\mathcal{B})^{\complement}}:=f_1(x)+f_2(x).
\end{equation*}
Then $\big[b,\mathcal{L}^{-\alpha/2}\big]^2(f)$ may be written as
\begin{equation*}
\begin{split}
\big[b,\mathcal{L}^{-\alpha/2}\big]^2(f)(x)&=\big[b(x)-b_{2\mathcal{B}}\big]^2\cdot \mathcal{L}^{-\alpha/2}(f)(x)
-2\big[b(x)-b_{2\mathcal{B}}\big]\cdot\mathcal{L}^{-\alpha/2}\big(\big[b-b_{2\mathcal{B}}\big]f\big)(x)\\
&+\mathcal{L}^{-\alpha/2}\big(\big[b-b_{2\mathcal{B}}\big]^2f_1\big)(x)+\mathcal{L}^{-\alpha/2}\big(\big[b-b_{2\mathcal{B}}\big]^2f_2\big)(x),
\end{split}
\end{equation*}
for any $x\in\mathbb R^n$. So we have
\begin{equation*}
\begin{split}
\mathbf{W}&\leq\frac{1}{m(\mathcal{B})^{{\gamma}/n}}\bigg(\frac{1}{m(\mathcal{B})}
\int_{\mathcal{B}}\Big|\big[b(x)-b_{2\mathcal{B}}\big]^2\cdot\mathcal{L}^{-\alpha/2}(f)(x)\Big|^{q}dx\bigg)^{1/{q}}\\
+&\frac{1}{m(\mathcal{B})^{{\gamma}/n}}\bigg(\frac{1}{m(\mathcal{B})}
\int_{\mathcal{B}}\Big|\mathcal{L}^{-\alpha/2}\big(\big[b-b_{2\mathcal{B}}\big]^2f_1\big)(x)\Big|^{q}dx\bigg)^{1/{q}}\\
+&\frac{1}{m(\mathcal{B})^{{\gamma}/n}}\bigg(\frac{1}{m(\mathcal{B})}
\int_{\mathcal{B}}\Big|e^{-t_{\mathcal{B}}\mathcal{L}}\big(\big[b-b_{2\mathcal{B}}\big]^2
\cdot\mathcal{L}^{-\alpha/2}(f)\big)(x)\Big|^{q}dx\bigg)^{1/{q}}\\
+&\frac{1}{m(\mathcal{B})^{{\gamma}/n}}\bigg(\frac{1}{m(\mathcal{B})}
\int_{\mathcal{B}}\Big|e^{-t_{\mathcal{B}}\mathcal{L}}\mathcal{L}^{-\alpha/2}
\big(\big[b-b_{2\mathcal{B}}\big]^2f_1\big)(x)\Big|^{q}dx\bigg)^{1/{q}}\\
+&\frac{1}{m(\mathcal{B})^{{\gamma}/n}}\bigg(\frac{1}{m(\mathcal{B})}
\int_{\mathcal{B}}\Big|2\big[b(x)-b_{2\mathcal{B}}\big]
\cdot\mathcal{L}^{-\alpha/2}\big(\big[b-b_{2\mathcal{B}}\big]f\big)(x)\Big|^{q}dx\bigg)^{1/{q}}\\
+&\frac{1}{m(\mathcal{B})^{{\gamma}/n}}\bigg(\frac{1}{m(\mathcal{B})}
\int_{\mathcal{B}}\Big|2e^{-t_{\mathcal{B}}\mathcal{L}}\Big(\big[b-b_{2\mathcal{B}}\big]
\mathcal{L}^{-\alpha/2}\big(\big[b-b_{2\mathcal{B}}\big]f\big)\Big)(x)\Big|^{q}dx\bigg)^{1/{q}}\\
+&\frac{1}{m(\mathcal{B})^{{\gamma}/n}}\bigg(\frac{1}{m(\mathcal{B})}
\int_{\mathcal{B}}\Big|\mathcal{L}^{-\alpha/2}\big(\big[b-b_{2\mathcal{B}}\big]^2f_2\big)(x)
-e^{-t_{\mathcal{B}}\mathcal{L}}\mathcal{L}^{-\alpha/2}\big(\big[b-b_{2\mathcal{B}}\big]^2f_2\big)(x)\Big|^{q}dx\bigg)^{1/{q}}\\
:=&\mathrm{I+II+III+IV+V+VI+VII}.
\end{split}
\end{equation*}
Here $t_{\mathcal{B}}=r_{\mathcal{B}}^2$ and $r_{\mathcal{B}}$ is the radius of the ball $\mathcal{B}$. Let us now estimate I, II, III, IV, V, VI and VII, respectively. Let $s$ be the same as in Theorem \ref{thm1}. Then
\begin{equation}\label{main412}
1/q=1/{p_1}+1/{p_1}+1/s.
\end{equation}
For the first term I, by H\"{o}lder's inequality and \eqref{main412}, we obtain
\begin{equation*}
\begin{split}
\mathrm{I}&\leq\frac{1}{m(\mathcal{B})^{{\gamma}/n+1/q}}\bigg(\int_{\mathcal{B}}\big|b(x)-b_{2\mathcal{B}}\big|^{p_1}dx\bigg)^{1/{p_1}}
\bigg(\int_{\mathcal{B}}\big|b(x)-b_{2\mathcal{B}}\big|^{p_1}dx\bigg)^{1/{p_1}}
\times\bigg(\int_{\mathcal{B}}\big|\mathcal{L}^{-\alpha/2}(f)(x)\big|^{s}dx\bigg)^{1/{s}}\\
&\leq\big\|b\big\|^2_{\mathcal{C}^{p_1,\beta_1}}\times
\frac{1}{m(\mathcal{B})^{{\gamma}/n+1/q-2\beta_1/n-2/{p_1}}}
\bigg(\int_{\mathcal{B}}\big|\mathcal{L}^{-\alpha/2}(f)(x)\big|^{s}dx\bigg)^{1/{s}}\\
&=\big\|b\big\|^2_{\mathcal{C}^{p_1,\beta_1}}\times\frac{1}{m(\mathcal{B})^{{(\gamma-2\beta_1)}/n}}
\bigg(\frac{1}{m(\mathcal{B})}\int_{\mathcal{B}}\big|\mathcal{L}^{-\alpha/2}(f)(x)\big|^{s}dx\bigg)^{1/{s}}.
\end{split}
\end{equation*}
Notice that
\begin{equation}\label{main413}
\gamma-2\beta_1=\beta_2+\alpha<0.
\end{equation}
This fact, together with Theorem \ref{generalizedmorrey}, implies that
\begin{equation*}
\begin{split}
\mathrm{I}\leq\big\|b\big\|^2_{\mathcal{C}^{p_1,\beta_1}}\big\|\mathcal{L}^{-\alpha/2}(f)\big\|_{\mathcal{M}^{s,\alpha+\beta_2}}
\leq C\big\|b\big\|^2_{\mathcal{C}^{p_1,\beta_1}}\big\|f\big\|_{\mathcal{M}^{p_2,\beta_2}}.
\end{split}
\end{equation*}
For the second term II, we now choose a real number $\widetilde{p}>1$ so that
\begin{equation*}
1/{\widetilde{p}}=1/{p_1}+1/{p_1}+1/{p_2}.
\end{equation*}
Then we have
\begin{equation}\label{main414}
1/q=1/{\widetilde{p}}-\alpha/n.
\end{equation}
By using H\"{o}lder's inequality and Theorem \ref{12}, we have
\begin{equation*}
\begin{split}
\mathrm{II}&\leq\frac{C}{m(\mathcal{B})^{{\gamma}/n+1/q}}\bigg(\int_{2\mathcal{B}}
\Big|\big[b(x)-b_{2\mathcal{B}}\big]^2\cdot f(x)\Big|^{\widetilde{p}}dx\bigg)^{1/{\widetilde{p}}}\\
&\leq\frac{C}{m(\mathcal{B})^{{\gamma}/n+1/q}}
\bigg(\int_{2\mathcal{B}}\big|b(x)-b_{2\mathcal{B}}\big|^{p_1}dx\bigg)^{1/{p_1}}
\bigg(\int_{2\mathcal{B}}\big|b(x)-b_{2\mathcal{B}}\big|^{p_1}dx\bigg)^{1/{p_1}}
\bigg(\int_{2\mathcal{B}}\big|f(x)\big|^{p_2}dx\bigg)^{1/{p_2}}.
\end{split}
\end{equation*}
Moreover, it follows from \eqref{main414} that
\begin{equation*}
\begin{split}
\mathrm{II}&\leq\frac{C}{m(\mathcal{B})^{{\gamma}/n+1/q}}
\cdot m(2\mathcal{B})^{2\beta_1/n+2/{p_1}}\cdot m(2\mathcal{B})^{\beta_2/n+1/{p_2}}
\big\|b\big\|^2_{\mathcal{C}^{p_1,\beta_1}}\big\|f\big\|_{\mathcal{M}^{p_2,\beta_2}}\\
&=\frac{C}{m(\mathcal{B})^{{\gamma}/n+1/q}}\cdot m(2\mathcal{B})^{{(2\beta_1+\beta_2+\alpha)}/n+1/{q}}
\big\|b\big\|^2_{\mathcal{C}^{p_1,\beta_1}}\big\|f\big\|_{\mathcal{M}^{p_2,\beta_2}}\\
&\leq C\big\|b\big\|^2_{\mathcal{C}^{p_1,\beta_1}}\big\|f\big\|_{\mathcal{M}^{p_2,\beta_2}},
\end{split}
\end{equation*}
where in the last step we have used the fact that $\gamma=2\beta_1+\beta_2+\alpha$.
Arguing as in the proof of Theorem \ref{thm1}, we can also prove that
\begin{equation*}
\mathrm{III+IV}\leq C\big\|b\big\|^2_{\mathcal{C}^{p_1,\beta_1}}\big\|f\big\|_{\mathcal{M}^{p_2,\beta_2}}.
\end{equation*}
Let us now turn to deal with the term V. It is easy to see that for any $x\in \mathcal{B}$,
\begin{equation*}
\big[b,\mathcal L^{-\alpha/2}\big](f)(x)=\big[b(x)-b_{2\mathcal{B}}\big]\cdot \mathcal{L}^{-\alpha/2}(f)(x)
-\mathcal{L}^{-\alpha/2}\big(\big[b-b_{2\mathcal{B}}\big]f\big)(x),
\end{equation*}
and then we obtain
\begin{equation}\label{second}
\begin{split}
&2\big[b(x)-b_{2\mathcal{B}}\big]
\cdot\mathcal{L}^{-\alpha/2}\big(\big[b-b_{2\mathcal{B}}\big]f\big)(x)\\
&=2\big[b(x)-b_{2\mathcal{B}}\big]^2\cdot \mathcal{L}^{-\alpha/2}(f)(x)
-2\big[b(x)-b_{2\mathcal{B}}\big]\cdot\big[b,\mathcal L^{-\alpha/2}\big](f)(x).
\end{split}
\end{equation}
Consequently, one can write
\begin{equation*}
\begin{split}
\mathrm{V}&\leq\frac{1}{m(\mathcal{B})^{{\gamma}/n}}\bigg(\frac{1}{m(\mathcal{B})}
\int_{\mathcal{B}}\Big|2\big[b(x)-b_{2\mathcal{B}}\big]^2
\cdot\mathcal{L}^{-\alpha/2}(f)(x)\Big|^{q}dx\bigg)^{1/{q}}\\
&+\frac{1}{m(\mathcal{B})^{{\gamma}/n}}\bigg(\frac{1}{m(\mathcal{B})}
\int_{\mathcal{B}}\Big|2\big[b(x)-b_{2\mathcal{B}}\big]
\cdot\big[b,\mathcal L^{-\alpha/2}\big](f)(x)\Big|^{q}dx\bigg)^{1/{q}}\\
&:=\mathrm{V}^{(1)}+\mathrm{V}^{(2)}.
\end{split}
\end{equation*}
By the estimate of I, we see that
\begin{equation*}
\mathrm{V}^{(1)}\leq C\big\|b\big\|^2_{\mathcal{C}^{p_1,\beta_1}}\big\|f\big\|_{\mathcal{M}^{p_2,\beta_2}}.
\end{equation*}
On the other hand, we choose a real number $\widetilde{q}>1$ so that
\begin{equation}\label{main415}
1/{\widetilde{q}}=1/{p_1}+1/{p_2}-\alpha/n\Longrightarrow 1/q=1/{p_1}+1/{\widetilde{q}}.
\end{equation}
It follows from H\"{o}lder's inequality and \eqref{main415} that
\begin{equation*}
\begin{split}
\mathrm{V}^{(2)}&\leq\frac{2}{m(\mathcal{B})^{{\gamma}/n+1/q}}\bigg(\int_{\mathcal{B}}\big|b(x)-b_{2\mathcal{B}}\big|^{p_1}dx\bigg)^{1/{p_1}}
\times\bigg(\int_{\mathcal{B}}\big|\big[b,\mathcal{L}^{-\alpha/2}\big](f)(x)\big|^{\widetilde{q}}dx\bigg)^{1/{\widetilde{q}}}\\
&\leq\frac{2}{m(\mathcal{B})^{{\gamma}/n+1/q}} \cdot m(2\mathcal{B})^{\beta_1/n+1/{p_1}}\big\|b\big\|_{\mathcal{C}^{p_1,\beta_1}}
\bigg(\int_{\mathcal{B}}\big|\big[b,\mathcal{L}^{-\alpha/2}\big](f)(x)\big|^{\widetilde{q}}dx\bigg)^{1/{\widetilde{q}}}\\
&\leq C\big\|b\big\|_{\mathcal{C}^{p_1,\beta_1}}\frac{1}{m(\mathcal{B})^{{(\gamma-\beta_1)}/n}}
\bigg(\frac{1}{m(\mathcal{B})}\int_{\mathcal{B}}\big|\big[b,\mathcal{L}^{-\alpha/2}\big](f)(x)\big|^{\widetilde{q}}dx\bigg)^{1/{\widetilde{q}}}.
\end{split}
\end{equation*}
Moreover, by using Theorem \ref{thm4}, we thus obtain
\begin{equation*}
\begin{split}
\mathrm{V}^{(2)}&\leq C\big\|b\big\|_{\mathcal{C}^{p_1,\beta_1}}
\big\|\big[b,\mathcal L^{-\alpha/2}\big](f)\big\|_{\mathcal{M}^{\widetilde{q},\beta_1+\beta_2+\alpha}}\\
&\leq C\big\|b\big\|^2_{\mathcal{C}^{p_1,\beta_1}}\big\|f\big\|_{\mathcal{M}^{p_2,\beta_2}}.
\end{split}
\end{equation*}
Summing up the estimates of $\mathrm{V}^{(1)}$ and $\mathrm{V}^{(2)}$, we conclude that
\begin{equation*}
\mathrm{V}\leq C\big\|b\big\|^2_{\mathcal{C}^{p_1,\beta_1}}\big\|f\big\|_{\mathcal{M}^{p_2,\beta_2}}.
\end{equation*}
As for the term VI, it follows from the estimates \eqref{main45} and \eqref{main46} that for any $x\in \mathcal{B}$,
\begin{equation*}
\begin{split}
&\Big|e^{-t_{\mathcal{B}}\mathcal{L}}\Big(\big[b-b_{2\mathcal{B}}\big]
\mathcal{L}^{-\alpha/2}\big(\big[b-b_{2\mathcal{B}}\big]f\big)\Big)(x)\Big|\\
&=\bigg|\int_{\mathbb R^n}\mathcal{P}_{t_{\mathcal{B}}}(x,y)
\cdot\big[b(y)-b_{2\mathcal{B}}\big]\mathcal{L}^{-\alpha/2}\big(\big[b-b_{2\mathcal{B}}\big]f\big)(y)\,dy\bigg|\\
&\leq C\frac{1}{m(2\mathcal{B})}\int_{2\mathcal{B}}
\Big|\big[b(y)-b_{2\mathcal{B}}\big]\cdot\mathcal{L}^{-\alpha/2}\big(\big[b-b_{2\mathcal{B}}\big]f\big)(y)\Big|\,dy\\
&+C\sum_{k=1}^{\infty}\frac{1}{2^{kn}}\frac{1}{m(2^{k+1}\mathcal{B})}\int_{2^{k+1}\mathcal{B}}
\Big|\big[b(y)-b_{2\mathcal{B}}\big]\cdot\mathcal{L}^{-\alpha/2}\big(\big[b-b_{2\mathcal{B}}\big]f\big)(y)\Big|\,dy.
\end{split}
\end{equation*}
As a consequence, we have
\begin{equation*}
\begin{split}
\mathrm{VI}&\leq\frac{C}{m(\mathcal{B})^{{\gamma}/n}}\bigg(\frac{1}{m(2\mathcal{B})}\int_{2\mathcal{B}}
\Big|\big[b(y)-b_{2\mathcal{B}}\big]\cdot\mathcal{L}^{-\alpha/2}\big(\big[b-b_{2\mathcal{B}}\big]f\big)(y)\Big|\,dy\bigg)\\
&+\frac{C}{m(\mathcal{B})^{{\gamma}/n}}\bigg(\sum_{k=1}^{\infty}\frac{1}{2^{kn}}\frac{1}{m(2^{k+1}\mathcal{B})}\int_{2^{k+1}\mathcal{B}}
\Big|\big[b(y)-b_{2\mathcal{B}}\big]\cdot\mathcal{L}^{-\alpha/2}\big(\big[b-b_{2\mathcal{B}}\big]f\big)(y)\Big|\,dy\bigg)\\
&:=\mathrm{VI}^{(1)}+\mathrm{VI}^{(2)}.
\end{split}
\end{equation*}
As in the estimate of V, we also obtain that
\begin{equation*}
\begin{split}
\mathrm{VI}^{(1)}&\leq\frac{C}{m(\mathcal{B})^{{\gamma}/n}}\bigg(\frac{1}{m(2\mathcal{B})}
\int_{2\mathcal{B}}\Big|\big[b(y)-b_{2\mathcal{B}}\big]
\cdot\mathcal{L}^{-\alpha/2}\big(\big[b-b_{2\mathcal{B}}\big]f\big)(y)\Big|^{q}dy\bigg)^{1/{q}}\\
&\leq C\big\|b\big\|^2_{\mathcal{C}^{p_1,\beta_1}}\big\|f\big\|_{\mathcal{M}^{p_2,\beta_2}}.
\end{split}
\end{equation*}
By using equation \eqref{second}, the term $\mathrm{VI}^{(2)}$ can be further divided into two parts.
\begin{equation*}
\begin{split}
\mathrm{VI}^{(2)}&\leq
\frac{C}{m(\mathcal{B})^{{\gamma}/n}}\bigg(\sum_{k=1}^{\infty}\frac{1}{2^{kn}}\frac{1}{m(2^{k+1}\mathcal{B})}\int_{2^{k+1}\mathcal{B}}
\Big|\big[b(y)-b_{2\mathcal{B}}\big]^2\cdot\mathcal{L}^{-\alpha/2}(f)(y)\Big|\,dy\bigg)\\
&+\frac{C}{m(\mathcal{B})^{{\gamma}/n}}\bigg(\sum_{k=1}^{\infty}\frac{1}{2^{kn}}\frac{1}{m(2^{k+1}\mathcal{B})}\int_{2^{k+1}\mathcal{B}}
\Big|\big[b(y)-b_{2\mathcal{B}}\big]\cdot\big[b,\mathcal L^{-\alpha/2}\big](f)(y)\Big|\,dy\bigg)\\
&:=\mathrm{VI}^{(3)}+\mathrm{VI}^{(4)}.
\end{split}
\end{equation*}
Applying Theorem \ref{thm4} again, we know that $\big[b,\mathcal L^{-\alpha/2}\big](f)$ belongs to the space $\mathcal{M}^{\widetilde{q},\beta_1+\beta_2+\alpha}(\mathbb R^n)$, and
\begin{equation*}
\big\|\big[b,\mathcal L^{-\alpha/2}\big](f)\big\|_{\mathcal{M}^{\widetilde{q},\beta_1+\beta_2+\alpha}}
\leq C\big\|b\big\|_{\mathcal{C}^{p_1,\beta_1}}\big\|f\big\|_{\mathcal{M}^{p_2,\beta_2}}.
\end{equation*}
Here the number $\widetilde{q}$ is the same as above. This fact together with Lemma \ref{wanglemma1} gives us that
\begin{equation*}
\begin{split}
\mathrm{VI}^{(4)}&\leq \frac{C}{m(\mathcal{B})^{{\gamma}/n}}\sum_{k=1}^{\infty}\frac{1}{2^{kn}}\cdot k
\Big[m\big(2^{k+1}\mathcal{B}\big)^{{(2\beta_1+\beta_2+\alpha)}/n}\cdot\big\|b\big\|_{\mathcal{C}^{p_1,\beta_1}}
\big\|\big[b,\mathcal L^{-\alpha/2}\big](f)\big\|_{\mathcal{M}^{\widetilde{q},\beta_1+\beta_2+\alpha}}\Big]\\
&\leq C\sum_{k=1}^{\infty}\frac{k}{2^{kn}}\left[\frac{m(2^{k+1}\mathcal{B})}{m(\mathcal{B})}\right]^{{\gamma}/n}
\big\|b\big\|^2_{\mathcal{C}^{p_1,\beta_1}}\big\|f\big\|_{\mathcal{M}^{p_2,\beta_2}}\\
&\leq C\sum_{k=1}^{\infty}\frac{k}{2^{k(n-\gamma)}}
\big\|b\big\|^2_{\mathcal{C}^{p_1,\beta_1}}\big\|f\big\|_{\mathcal{M}^{p_2,\beta_2}}
\leq C\big\|b\big\|^2_{\mathcal{C}^{p_1,\beta_1}}\big\|f\big\|_{\mathcal{M}^{p_2,\beta_2}}.
\end{split}
\end{equation*}
Similarly, we obtain that
\begin{equation*}
\mathrm{VI}^{(3)}\leq C\big\|b\big\|^2_{\mathcal{C}^{p_1,\beta_1}}\big\|f\big\|_{\mathcal{M}^{p_2,\beta_2}}.
\end{equation*}
Summing up the above estimates, we conclude that
\begin{equation*}
\mathrm{VI}\leq C\big\|b\big\|^2_{\mathcal{C}^{p_1,\beta_1}}\big\|f\big\|_{\mathcal{M}^{p_2,\beta_2}}.
\end{equation*}
It remains to estimate the term VII. Applying Lemma \ref{wanglemma2}, we obtain that for any $x\in \mathcal{B}$,
\begin{equation*}
\begin{split}
&\Big|\mathcal{L}^{-\alpha/2}\big(\big[b-b_{2\mathcal{B}}\big]^2f_2\big)(x)
-e^{-t_{\mathcal{B}}\mathcal{L}}\mathcal{L}^{-\alpha/2}\big(\big[b-b_{2\mathcal{B}}\big]^2f_2\big)(x)\Big|\\
&=\Big|\big(I-e^{-t_{\mathcal{B}}\mathcal{L}}\big)\mathcal{L}^{-\alpha/2}\big(\big[b-b_{2\mathcal{B}}\big]^2f_2\big)(x)\Big|\\
&\leq\int_{(2\mathcal{B})^{\complement}}\big|\widetilde{K}_{\alpha,t_{\mathcal{B}}}(x,y)\big|
\cdot\big|\big[b(y)-b_{2\mathcal{B}}\big]^2f(y)\big|\,dy\\
&\leq C\sum_{k=1}^{\infty}\int_{2^{k+1}\mathcal{B}\setminus 2^k \mathcal{B}}
\frac{1}{|x-y|^{n-\alpha}}\cdot\frac{r_{\mathcal{B}}^2}{|x-y|^{2}}\big|b(y)-b_{2\mathcal{B}}\big|^2\big|f(y)\big|\,dy\\
&\leq C\sum_{k=1}^{\infty}\frac{1}{2^{2k}}\cdot\frac{1}{m(2^{k+1}\mathcal{B})^{1-\alpha/n}}
\int_{2^{k+1}\mathcal{B}}\big|b(y)-b_{2\mathcal{B}}\big|^2\big|f(y)\big|\,dy.
\end{split}
\end{equation*}
Moreover, by using the same arguments as in the proof of Lemma \ref{wanglemma1}, we can prove that the following estimate
\begin{equation*}
\frac{1}{m(2^{k+1}\mathcal{B})}\int_{2^{k+1}\mathcal{B}}
\big|b(y)-b_{2\mathcal{B}}\big|^2\big|f(y)\big|\,dy
\leq C\cdot k\Big[m\big(2^{k+1}\mathcal{B}\big)^{{(2\beta_1+\beta_2)}/n}
\cdot\big\|b\big\|^2_{\mathcal{C}^{p_1,\beta_1}}\big\|f\big\|_{\mathcal{M}^{p_2,\beta_2}}\Big]
\end{equation*}
holds for any $f\in \mathcal{M}^{p_2,\beta_2}(\mathbb R^n)$ with $1<p_2<n/{\alpha}$ and $-n/{p_2}\leq\beta_2<(-\alpha)$, and $b\in \mathcal{C}^{p_1,\beta_1}(\mathbb R^n)$ with $1\leq p_1<\infty$ and $-n/{p_1}\leq\beta_1<0$. This in turn implies that
\begin{equation*}
\begin{split}
\mathrm{VII}&=\frac{1}{m(\mathcal{B})^{{\gamma}/n}}\bigg(\frac{1}{m(\mathcal{B})}
\int_{\mathcal{B}}\Big|\mathcal{L}^{-\alpha/2}\big(\big[b-b_{2\mathcal{B}}\big]^2f_2\big)(x)
-e^{-t_{\mathcal{B}}\mathcal{L}}\mathcal{L}^{-\alpha/2}\big(\big[b-b_{2\mathcal{B}}\big]^2f_2\big)(x)\Big|^{q}dx\bigg)^{1/{q}}\\
&\leq \frac{C}{m(\mathcal{B})^{{\gamma}/n}}\sum_{k=1}^{\infty}\frac{1}{2^{2k}}\cdot\frac{1}{m(2^{k+1}\mathcal{B})^{1-\alpha/n}}
\int_{2^{k+1}\mathcal{B}}\big|b(y)-b_{2\mathcal{B}}\big|^2\big|f(y)\big|\,dy\\
&\leq \frac{C}{m(\mathcal{B})^{{\gamma}/n}}
\sum_{k=1}^{\infty}\frac{k}{2^{2k}}\cdot m\big(2^{k+1}\mathcal{B}\big)^{{(2\beta_1+\beta_2+\alpha)}/n}
\big\|b\big\|^2_{\mathcal{C}^{p_1,\beta_1}}\big\|f\big\|_{\mathcal{M}^{p_2,\beta_2}}.\\
\end{split}
\end{equation*}
Therefore, by the facts that $\gamma=2\beta_1+\beta_2+\alpha$ and $\gamma<0$, we further obtain
\begin{equation*}
\begin{split}
\mathrm{VII}&\leq C\sum_{k=1}^{\infty}\frac{k}{2^{2k}}\cdot\left[\frac{m(2^{k+1}\mathcal{B})}{m(\mathcal{B})}\right]^{{\gamma}/n}
\big\|b\big\|^2_{\mathcal{C}^{p_1,\beta_1}}\big\|f\big\|_{\mathcal{M}^{p_2,\beta_2}}\\
&\leq C\sum_{k=1}^{\infty}\frac{k}{2^{(2-\gamma)k}}
\big\|b\big\|^2_{\mathcal{C}^{p_1,\beta_1}}\big\|f\big\|_{\mathcal{M}^{p_2,\beta_2}}\\
&\leq C\big\|b\big\|^2_{\mathcal{C}^{p_1,\beta_1}}\big\|f\big\|_{\mathcal{M}^{p_2,\beta_2}}.
\end{split}
\end{equation*}
Combining the estimates of I, II, III, IV, V, VI with VII, we get the desired result \eqref{main411}. Hence, the proof of Theorem \ref{thm2} is
complete.
\end{proof}

By induction on $m$, we can also prove the following general result.
\begin{theorem}\label{thm5}
Let $2\leq m\in \mathbb{N}$, $0<\alpha<n$, $1<p_2<n/{\alpha}$ and $-n/{p_2}\leq\beta_2<(-\alpha)$. Suppose that $b\in \mathcal{C}^{p_1,\beta_1}(\mathbb R^n)$ with $1\leq p_1<\infty$ and $-n/{p_1}\leq\beta_1<0$. Then for any $f\in \mathcal{M}^{p_2,\beta_2}(\mathbb R^n)$, there exists a positive constant $C>0$ independent of $b$ and $f$ such that
\begin{equation*}
\big\|\big[b,\mathcal{L}^{-\alpha/2}\big]^m(f)\big\|_{\mathcal{C}^{q,\gamma}_{\mathcal{L}}}
\leq C\big\|b\big\|^m_{\mathcal{C}^{p_1,\beta_1}}\big\|f\big\|_{\mathcal{M}^{p_2,\beta_2}},
\end{equation*}
provided that
\begin{equation*}
1/q=m/{p_1}+1/{p_2}-\alpha/n\quad \& \quad \gamma=m\beta_1+\beta_2+\alpha.
\end{equation*}
\end{theorem}

In particular, we have
\begin{corollary}\label{thm6}
Let $2\leq m\in \mathbb{N}$ and $0<\alpha<n$. Suppose that $b\in \mathcal{C}^{p_1,\beta_1}(\mathbb R^n)$ with $1\leq p_1<\infty$ and $-n/{p_1}\leq\beta_1<0$. Then for any
$f\in L^{p_2}(\mathbb R^n)$ with $1<p_2<n/{\alpha}$, there exists a positive constant $C>0$ independent of $b$ and $f$ such that
\begin{equation*}
\big\|\big[b,\mathcal L^{-\alpha/2}\big]^m(f)\big\|_{\mathcal{C}^{q,\gamma}_{\mathcal{L}}}
\leq C\big\|b\big\|^m_{\mathcal{C}^{p_1,\beta_1}}\big\|f\big\|_{L^{p_2}},
\end{equation*}
provided that
\begin{equation*}
1/q=m/{p_1}+1/{p_2}-\alpha/n\quad \& \quad \gamma=m\beta_1+\alpha-n/{p_2}.
\end{equation*}
\end{corollary}
In fact, Corollary \ref{thm6} is a straightforward consequence of Theorem \ref{thm5}, since $\mathcal{M}^{p_2,\beta_2}(\mathbb R^n)=L^{p_2}(\mathbb R^n)$ if we take $\beta_2=-n/{p_2}$.

We remark that when $q\geq1$ and $2\leq m\in \mathbb{N}$,
\begin{equation*}
\begin{split}
&1/q=m/{p_1}+1/{p_2}-\alpha/n\Longrightarrow m\leq p_1.\\
&\gamma=m\beta_1+\beta_2+\alpha\Longrightarrow -n/{q}\leq\gamma<0.
\end{split}
\end{equation*}
\begin{remark}
Let $2\leq m\in \mathbb{N}$ and $b(x)$ be a locally integrable function on $\mathbb R^n$. The higher-order commutator $\big[b,I_{\alpha}\big]^m$ generated by $I_{\alpha}$ and $b$ is defined by
\begin{equation*}
\big[b,I_{\alpha}\big]^m(f)(x):=\big[b,\dots\big[b,\big[b,I_{\alpha}\big]\big]\big](f)(x),\quad 0<\alpha<n.
\end{equation*}
That is, these commutators $\big[b,I_{\alpha}\big]^m$($m=1,2,\dots$) can be defined by recurrence:
\begin{equation*}
\big[b,I_{\alpha}\big]^m(f)=\big[b,\big[b,I_{\alpha}\big]^{m-1}(f)\big],
\end{equation*}
where
\begin{equation*}
\big[b,I_{\alpha}\big]^1(f):=\big[b,I_{\alpha}\big](f).
\end{equation*}
Then we have
\begin{equation*}
\big[b,I_{\alpha}\big]^m(f)(x)=\frac{1}{\gamma(\alpha)}\int_{\mathbb R^n}\frac{[b(x)-b(y)]^m}{|x-y|^{n-\alpha}}f(y)\,dy.
\end{equation*}
Following \cite{shi1,shi2}, we will say that a locally integrable function $b(x)$ belongs to the reverse H\"{o}lder class $RH_{\infty}$, if there exists a constant $C>0$ such that for any ball $\mathcal{B}\subset\mathbb R^n$,
\begin{equation}\label{reverseh}
\sup_{x\in \mathcal{B}}\big|b(x)-b_{\mathcal{B}}\big|
\leq C\bigg(\frac{1}{m(\mathcal{B})}\int_{\mathcal{B}}\big|b(x)-b_{\mathcal{B}}\big|\,dx\bigg).
\end{equation}
When $L=-\Delta$, from Theorem \ref{thm5} and Corollary \ref{thm6}, we can obtain the corresponding results for $\big[b,I_{\alpha}\big]^m$, which have been proved by Shi and Lu, under the assumption that $b(x)$ satisfies \eqref{reverseh}, see \cite[Theorem 1.1]{shi2}. It should be pointed out that the condition \eqref{reverseh} assumed on symbol functions has been removed from Theorems \ref{thm2} and \ref{thm5} in this paper. We improve and extend Shi and Lu's result \cite{shi2} about the higher-order commutator $\big[b,I_{\alpha}\big]^m$ with $2\leq m\in \mathbb{N}$. 
\end{remark}

{\bf Acknowledgments.}

The authors were supported by a grant from Xinjiang University under the project``Real-Variable Theory of Function Spaces and Its Applications".This work was supported by the Natural Science Foundation of China (No.XJEDU2020Y002 and 2022D01C407).

\bibliographystyle{amsplain}

\begin{thebibliography}{99}

\bibitem{adams}
D. R. Adams, \emph{A note on Riesz potentials}, Duke Math. J, \textbf{42}(1975), 765--778.
\bibitem{adams1}
D. R. Adams,\emph{Morrey Spaces},Lecture notes in applied and numerical harmonic analysis, Birkh\"{a}user/Springer, Cham, 2015.
\bibitem{auscher}
P. Auscher and J. M. Martell, \emph{Weighted norm inequalities for fractional operators}, Indiana Univ. Math. J, \textbf{57}(2008), 1845--1870.
\bibitem{benyi}
A. Benyi, J. M. Martell, K. Moen, E. Stachura and R. H. Torres, \emph{Boundedness results for commutators with BMO functions via weighted estimates: a comprehensive approach},  Math. Ann. \textbf{376}(2020), 61--102.
\bibitem{cam}
S.Campanato,\emph{Propriet\`{a}di H\"{o}lderianit\`{a} di alcune classi di funzioni}, Ann. Scuola Norm. Sup. Pisa., \textbf{17} (1963), 173--188.
\bibitem{cha}
S. Chanillo, \emph{A note on commutators}, Indiana Univ. Math. J, \textbf{31}(1982), 7--16.
\bibitem{cruz}
D. Cruz-Uribe, J. M. Martell and C. P\'erez, \emph{Extrapolation from $A_\infty$ weights and applications}, J. Funct. Anal, \textbf{213}(2004), 412--439.
\bibitem{deng2}
D. G. Deng, X. T. Duong and L. X. Yan, \emph{A characterization of the Morrey--Campanato spaces}, Math. Z., \textbf{250} (2005), 641--655.
\bibitem{deng}
D. G. Deng, X. T. Duong, A. Sikora and L. X. Yan, \emph{Comparison of the classical BMO with the BMO spaces associated with operators and applications}, Rev. Mat. Iberoamericana, \textbf{24}(2008), 267--296.
\bibitem{duong1}
X. T. Duong and L. X. Yan, \emph{On commutators of fractional integrals}, Proc. Amer. Math. Soc, \textbf{132}(2004), 3549--3557.
\bibitem{duong3}
X. T. Duong and L. X. Yan, \emph{New function spaces of BMO type, the John--Nirenberg inequality, interpolation and applications}, Comm. Pure Appl. Math, \textbf{58}(2005), 1375--1420.
\bibitem{duong2}
X. T. Duong and L. X. Yan, \emph{Duality of Hardy and BMO spaces associated with operators with heat kernel bounds}, J. Amer. Math. Soc, \textbf{18}(2005), 943--973.
\bibitem{grafakos}
L. Grafakos, \emph{Modern Fourier Analysis}, Third Edition, Springer-Verlag, 2014.
\bibitem{janson}
S. Janson, M. H. Taibleson and G. Weiss, \emph{Elementary characterizations of the Morrey--Campanato spaces}, Lecture Notes in Math., \textbf{992} (1983), 101--114.
\bibitem{john}
F. John and L. Nirenberg, \emph{On functions of bounded mean oscillation}, Comm. Pure Appl. Math, \textbf{14}(1961), 415--426.
\bibitem{komori}
Y. Komori and S. Shirai, \emph{Weighted Morrey spaces and a singular integral operator}, Math. Nachr, \textbf{282}(2009), 219--231.
\bibitem{lu}
S. Z. Lu, Y. Ding and D. Y. Yan, \emph{Singular Integrals and Related Topics}, World Scientific Publishing, NJ, 2007.
\bibitem{martell}
J. M. Martell, \emph{Sharp maximal functions associated with approximations of the identity in spaces of homogeneous type and applications}, Studia Math, \textbf{161}(2004), 113--145.
\bibitem{mo}
H. X. Mo and S. Z. Lu, \emph{Boundedness of multilinear commutators of generalized fractional integrals}, Math. Nachr, \textbf{281}(2008), 1328--1340.
\bibitem{morrey}
C. B. Morrey, \emph{On the solutions of quasi-linear elliptic partial differential equations}, Trans. Amer. Math. Soc., \textbf{43}(1938), 126--166.
\bibitem{muckenhoupt1}
B. Muckenhoupt and R. L. Wheeden, \emph{Weighted norm inequalities for singular and fractional integrals}, Trans. Amer. Math. Soc., \textbf{161}(1971), 249--258.
\bibitem{muckenhoupt2}
B. Muckenhoupt and R. L. Wheeden, \emph{Weighted norm inequalities for fractional integrals}, Trans. Amer. Math. Soc., \textbf{192}(1974), 261--274.
\bibitem{peetre}
J. Peetre, \emph{On the theory of $\mathcal L_{p,\lambda}$ spaces}, J. Funct. Anal, \textbf{4}(1969), 71--87.
\bibitem{paluszynski}
M. Paluszy\'{n}ski, \emph{Characterization of the Besov spaces via the commutator operator of Coifman, Rochberg and Weiss}, Indiana Univ. Math. J, \textbf{44}(1995), 1--17.
\bibitem{segovia}
C. Segovia and J. L. Torrea, \emph{Weighted inequalities for commutators of fractional and singular integrals}, Publ. Mat, \textbf{35}(1991), 209--235.
\bibitem{shi1}
S. G. Shi and S. Z. Lu, \emph{Some characterizations of Campanato spaces via commutators on Morrey spaces}, Pacific J. Math., \textbf{264} (2013), 221--234.
\bibitem{shi2}
S. G. Shi and S. Z. Lu, \emph{A characterization of Campanato space via commutator of fractional integral}, J. Math. Anal. Appl., \textbf{419} (2014), 123--137.
\bibitem{stein}
E. M. Stein, \emph{Singular Integrals and Differentiability Properties of Functions}, Princeton Univ. Press, Princeton, New Jersey, 1970.
\bibitem{wang}
H. Wang, \emph{Boundedness of fractional integral operators with rough kernels on weighted Morrey spaces}, Acta Math. Sinica (Chin. Ser),  \textbf{56}(2013), 175--186.
\bibitem{wang2}
H. Wang, \emph{Some estimates for commutators of fractional integrals associated to operators with Gaussian kernel bounds on weighted Morrey spaces}, Anal. Theory Appl., \textbf{29}(2013), 72--85.
\bibitem{wang3}
H. Wang, \emph{Estimates for fractional integral operators and linear commutators on certain weighted amalgam spaces}, J. Funct. Spaces 2020, Art. ID 2697104, 25 pp.
\end{thebibliography}

\end{document}